\documentclass[12pt]{article}
\usepackage[T2A]{fontenc}
\usepackage[utf8]{inputenc}
\usepackage{graphicx}
\usepackage{placeins}
\usepackage{bmpsize}
\usepackage{amsfonts, amsmath, amssymb, amsthm}
\usepackage{floatflt}
\usepackage{graphics}
\usepackage{enumitem}
\usepackage{hyperref}
\usepackage[usenames]{color}
\usepackage{colortbl}

\DeclareGraphicsExtensions{.pdf,.png,.jpg}

\theoremstyle{plain}
\newtheorem{theorem}{Theorem} 
\newtheorem{lemma}{Lemma} 
\newtheorem*{corollary}{Corollary}
 
\newtheorem{claim}{Claim}

\theoremstyle{remark}

\newtheorem*{remark}{Remark}

\theoremstyle{definition}
\newtheorem*{defin}{Definition}

\def\J{\mathcal{J}}
\def\G{\mathcal{G}}
\def\EHR{\mathrm{EHR}}
\def\Nei{\mathcal{N}}

\def\geqslant{\geq}
\def\leqslant{\leq}

\def\N{\mathbb{N}}

\def\set#1{\left\{#1\right\}}

\def\p{\mathsf{P}}

\begin{document}

\title{Spectrum of FO logic with quantifier depth 4 is finite}

\author{Yury Yarovikov\footnote{Moscow Institute of Physics and Technology, Dolgoprudny, Russia; Artificial Intelligence Research Institute (AIRI), Moscow, Russia. This work was supported by a grant for research centers in the field of artificial intelligence, provided by the Analytical Center for the Government of the Russian Federation in accordance with the subsidy agreement (agreement identifier 000000D730321P5Q0002) and the agreement with the Moscow Institute of Physics and Technology dated November 1, 2021 No. 70-2021-00138.}, Maksim Zhukovskii\footnote{The University of Sheffield}}

\date{\vspace{-5ex}}

\maketitle

\sloppy 

\newcounter{mygpropertiescounter}
\setcounter{mygpropertiescounter}{0}
\def\gpr{\par\bigskip\refstepcounter{mygpropertiescounter}\textbf{\arabic{mygpropertiescounter}.} }
\renewcommand{\themygpropertiescounter}{\arabic{mygpropertiescounter}}

\begin{abstract}
     The $k$-spectrum is the set of all $\alpha>0$ such that $G(n,n^{-\alpha})$ does not obey the 0-1 law for FO sentences with quantifier depth at most $k$. In this paper, we prove that the minimum $k$ such that the $k$-spectrum is infinite equals 5.
     
\end{abstract}

\section{Introduction}
\label{intro}

In this paper, we study an asymptotic behavior of probabilities of properties of Erd\H{o}s--R\'{e}nyi random graphs $G(n,p=p(n))$ expressible as sentences in the first order (FO) logic. Recall that $G(n,p)$ is a random element of the set of all simple graphs $\Omega_n = \set{\mathcal{G} = ([n]=\{1,\ldots,n\}, \mathcal{E})}$, $n\in\mathbb{N}$,  with the distribution ${\sf P}(G(n,p)=\mathcal{G}) = p^{|\mathcal{E}|} (1 - p)^{{n \choose 2} - |\mathcal{E}|}$, i.e. every pair of distinct vertices is adjacent with probaility $p\in[0,1]$ independently of all the others. 

A graph {\it FO property} $\mathcal{L}$ is a property defined by a FO sentence $\varphi$~\cite{libkin} in the vocabulary $\set{=, \sim}$ consisting of two predicate symbols, ``$=$'' expressing coincidence of vertices and ``$\sim$'' expressing adjacency (i.e. $G$ has the property $\mathcal{L}$ if and only if $G$ satisfies $\varphi$).
Recall that the {\it quantifier depth} of a FO sentence $\varphi$ is, roughly speaking, the minimum number of nested quantifiers (see the formal definition in \cite[Definition 3.8]{libkin}). We define the {\it quantifier depth of a FO property} as the minimal quantifier depth of a FO sentence expressing this property.

Consider, for example, the FO sentence
$$
\forall x \exists y \exists z \quad (x\sim y \wedge y \sim z \wedge [\exists t \,\,x \sim z]).
$$
This formula is true if and only if each vertex of the graph is contained in a triangle. The quantifier depth of this formula equals $4$. The quantifier depth of the property expressed by the formula, however, is no more than $3$ (and, in fact, equals 3), since the quantifier over $t$ can be safely removed.\\

Note that many natural graph properties either hold with asymptotical probability 1 (or, briefly, `a.a.s.' meaning `asymptotically almost surely') on the random graph chosen uniformly at random from the set of all graphs on $[n]$ (i.e. $G(n,1/2)$), or do not hold a.a.s.; and the same is true for $G(n,p)$ for every constant probability $p$. In particular, for every $p=\mathrm{const}\in(0,1)$ a.a.s. $G(n,p)$ is connected, contains a triangle, is not planar, has diameter at most 2, is Hamiltonian, etc. All the mentioned properties belong to the class of so called monotone properties (for formal definitions, see Section~\ref{known-definitions}).  \textcolor{black}{In particular, all existential positive FO properties are monotone.} 
 Much more is known about monotone properties: every monotone property $\mathcal{L}$ has a threshold probability $p_0$, i.e. $G(n,p)$ has the property $\mathcal{L}$ a.a.s. if $p_0=o(p)$, and $G(n,p)$ does not have the property $\mathcal{L}$ a.a.s. if $p=o(p_0)$, or vice versa.  

The first attempt to formalise the set of properties that either hold a.a.s. on $G(n,p)$ or do not hold a.a.s. was done in a seminal paper of Glebskii, Kogan, Liogon’kii and Talanov in 1969~\cite{glebskii}: they proved (and independently Fagin in 1976~\cite{fagin}) that, for every FO property, either a.a.s $G(n,1/2)$ has it, or a.a.s. does not have; or in other words $G(n,1/2)$ {\it obeys FO 0-1 law}.\\

\begin{defin}
$G(n,p)$ \textit{obeys FO 0-1 law} if, for each FO property $\mathcal{L}$, 
$$
\lim_{n\to\infty}{\sf P}(G(n,p)\in\mathcal{L})\in\{0,1\}.
$$
\end{defin}

Their result can be easily extended (as was noticed by Spencer~\cite{Spencer_ehren}) to all $p$ such that, for every $\alpha>0$, $pn^{\alpha}\to\infty$ as $n\to\infty$. It is natural to ask, does FO 0-1 law hold when $p=n^{-\alpha}$? Shelah and Spencer in 1988~\cite{spencershelah} gave a complete answer.\\


\begin{theorem}[S. Shelah, J. Spencer, 1988, \cite{spencershelah}] \label{alpha irrational} Let $\alpha>0$, $p = n^{-\alpha}$. If $\alpha$ is irrational, then $G(n,p)$ obeys FO 0-1 law. If $\alpha\leq 1$ is rational, then $G(n,p)$ does not obey FO 0-1 law. If $\alpha>1$, then $G(n,p)$ does not obey the law if and only if $\alpha=1+\frac{1}{m}$ for some positive integer $m$.\\
\end{theorem}

\medskip

From Theorem~\ref{alpha irrational} it immediately follows that, for any FO property $\mathcal{L}$, the set of $\alpha>0$ such that $\p(G(n,n^{-\alpha})\in\mathcal{L})$ does not approach neither 0 nor 1, consists only of rational numbers.

\begin{defin} 
For a FO property $\mathcal{L}$, the \textit{spectrum} of $\mathcal{L}$ is the set of all $\alpha > 0$ such that $\p(G(n, n^{-\alpha}) \in \mathcal{L})$ does not tend to either $0$ or $1$ as $n\to\infty$. \\
\end{defin}

Note that, for monotone properties, spectra may consist of at most 1 element. By combinining several monotone properties via logical connectives, we may reach only finite spectra. It is natural to expect that for all FO properties their spectra are finite. However, this is not true: sentences with infinite spectra can be found in~\cite{spencershelah,spencer_inf}. In this paper, we study a complexity of FO properties with infinite spectra. More precisely, we find the minimum quantifier depth of a property with an infinite spectrum.\\

Let $k\in\mathbb{N}$. Define $\mathrm{FO}_k$ as the set of all FO properties with quantifier depth at most $k$. It is well known that the \textcolor{black}{hierarchy} $\mathrm{FO}_k\subset\mathrm{FO}_{k+1}$ is strict for all $k$ (e.g., containing a $(k+1)$-clique belongs to $\mathrm{FO}_{k+1}\setminus\mathrm{FO}_k$), and that $\mathrm{FO}_k$ is expressive even for small values of $k$ (in particular, the property of containing an isolated vertex has quantifier depth 2).  The notion of quantifier depth of a formula lies on a bridge between logical and computational complexities. One can show, for instance, that, for a given property $\mathcal{L} \in FO_k$, there exists an algorithm that determines wether a graph $G$ with $v(G) = n$ has $\mathcal{L}$ in time $O(n^k)$ (see~\cite[Chapter 6]{libkin}).\\ 

\begin{defin}
$G(n,p)$ \textit{obeys $\mathrm{FO}_k$ 0-1 law} if for each FO property $\mathcal{L}$ with quantifier depth at most $k$, $\lim_{n\to\infty}{\sf P}(G(n,p)\in\mathcal{L})\in\{0,1\}$.\\
\end{defin}

Note that the study of validity of \textcolor{black}{$\mathrm{FO}_k$ 0-1 law} have direct corollaries to an estimation of expressive powers of fragments of FO logic. In particular, Theorem~\ref{zhukovskiiless} implies that the minimum quantifier depth of a FO sentence describing the property of containing an induced subgraph isomorphic to a given graph $F$ is at least $\frac{|E(F)|}{|V(F)|}+2$ (see~\cite{verbitsky}~and~\cite[Chapter 6]{verbitsky_induced}).

The obvious corollary of Theorem \ref{alpha irrational} is that, for functions $p = n^{-\alpha}$, $\alpha\notin\mathbb{Q}$, \textcolor{black}{$\mathrm{FO}_k$ 0-1 law} also holds. For $\alpha\in\mathbb{Q}$, the situation becomes more diverse. In particular, the following is true.\\

\begin{theorem}[M. E. Zhukovskii, 2012, \cite{zhukovskii}] \label{zhukovskiiless}
Let $p=n^{-\alpha}$, $\alpha \in \left(0, \frac{1}{k-2}\right)$, where $k \geqslant 3$ is an integer number. Then $G(n,p)$ obeys $\mathrm{FO}_k$ 0-1 law. If $\alpha = \frac{1}{k-2}$, then $G(n,p)$ does not obey $\mathrm{FO}_k$ 0-1 law.\\
\end{theorem}

{\color{black} It is very well known that, for every $k$ there are only finitely many properties with quantifier depth $k$ (see, e.g.,~\cite[Chapter 3]{libkin}). Therefore, a FO property with quantifier depth $k$ has an infinite spectrum if and only if the union of spectra of FO properties with quantifier depth $k$ is infinite. Note that such a union is exactly the set of all $\alpha>0$ such that $G(n,n^{-\alpha})$ does not obey $\mathrm{FO}_k$ 0-1 law.




\begin{defin}
For a fixed integer $k$, the {\it $k$-spectrum} is the set of all $\alpha > 0$ such that $G(n,n^{-\alpha})$ does not obey FO$_k$ 0-1 law.\\
\end{defin}



Note that Theorem \ref{zhukovskiiless} states that, for $k \geqslant 3$, the minimal element of the $k$-spectrum equals $\frac{1}{k-2}$.


In 1990~\cite{spencer_inf},  Spencer proved that the 14-spectrum is infinite. 
In particular, Spencer proved the existence of a FO$_{14}$ property $\mathcal{L}$ with the spectrum having $1/3$ as a limit point (i.e. each neighbourhood of $1/3$ contains an infinite number of points of the set). Note that, since the $k$-spectrum is a bounded set (in particular, it is entirely inside $(0,2]$ due to Theorem~\ref{alpha irrational} -- indeed, all points of spectra that are greater than 1 equal to $1+1/m$ for some positive integer $m$), the $k$-spectrum is finite if and only if it has no limit points. }

By Theorem \ref{alpha irrational}, limit points of the $k$-spectrum belong to $(0, 1]$. The behavior of limit points of the $k$-spectrum near 0 and 1 for large enough $k$ was studied in~\cite{SZ}.\\

As we note above, the existence of FO properties with an infinite spectrum is rather counterintuitive. It might come as a bit of relief that any limit point of the $k$-spectrum is a limit point ``from above'': for each $\alpha > 0$ there exists $\varepsilon > 0$ such that $(\alpha - \varepsilon, \alpha)$ does not contain any points of the $k$-spectrum \cite[Section 8.4]{strangelogic}.\\

In 2016~\cite{smallk}, the second author proved that $\frac12$ is a limit point of the $5$-spectrum, thus establishing that the $5$-spectrum is infinite. Moreover, the $3$-spectrum is finite by Theorem \ref{zhukovskiiless} and the fact that $1$ is not a limit point of a $k$-spectrum (see \cite[Theorem 8]{ostrovskii}).

Whether the 4-spectrum is finite was an open question before this paper. The following results were obtained recently.\\

\begin{theorem}[A. D. Matushkin, M. E. Zhukovskii, 2017, \cite{matushkin}]
The only limit points of the $4$-spectrum may be $\frac12$ and $\frac35$.\\
\end{theorem}

\begin{theorem}[Y. N. Yarovikov, 2021, \cite{yu-rovikov}] \label{1/2} $\frac{1}{2}$ is not a limit point of the 4-spectrum.\\
\end{theorem}

In this paper, we consider the last possible limit point of the $4$-spectrum, namely, $\frac35$. The main result of the paper is as follows.\\

\begin{theorem} 
$\frac35$ is not a limit point of the $4$-spectrum.\\  
\label{main_theorem} 
\end{theorem}

Therefore, the minimum $k$ such that the $k$-spectrum is infinite equals $5$. Note that this immediately implies that the $k$-spectrum is infinite for all $k\geq 5$ and finite for all $k\leq 4$.\\ 

Among the surveys on asymptotic behaviour of logical properties, we recommend~\cite{strangelogic}~and~\cite[Chapter 4]{finite-model-theory}. For going deeper into the subject of finite model theory and, in particular, descriptive complexity, we suggest \cite{immerman,libkin}.\\

\textcolor{black}{An upper level scheme of the proof of Theorem~\ref{main_theorem} is similar to that of Theorem~\ref{1/2} proven in~\cite{yu-rovikov}. However, the proof itself has been significantly modified. For convenience of a reader, the text of this paper is self-contained --- the only exclusion is the absense of the proof of Claim~\ref{helpKTokralphasafe} (from Section~\ref{known-definitions}) which appears in full generality in~\cite{yu-rovikov}. In Section~\ref{known-definitions} we recall some known constructions needed to prove Theorem~\ref{main_theorem}. Note that all definitions and theorems presented in Section~\ref{known-definitions} are rather standard, for proofs we refer a reader for most relevant sources. 
 In Section~\ref{sc:new_constructions}, we introduce auxiliary constructions inspired, in part, by~\cite{yu-rovikov}. However, all the claims stated in Section~\ref{sc:new_constructions} are new, and so we present their proofs in full. In Section~\ref{sc:3_lemmata}, we prove three main lemmas that are also inspired by~\cite{yu-rovikov} but their proofs are much more involved. Finally, we prove Theorem~\ref{main_theorem} in Section~\ref{sc:th_proof}.\\}






\section{Known definitions and statements}\label{known-definitions}
\subsection{Ehrenfeucht game}

In Ehrenfeucht game $\EHR(X, Y, k)$ on graphs $X$, $Y$, there are two players, Spoiler and Duplicator, and a fixed number of rounds $k\in\mathbb{N}$ (\cite{ehrenfeucht}). Each round of the game constitutes a move by Spoiler followed by a move by Duplicator: Spoiler selects a vertex from \textit{either} of the two graphs $X$ and $Y$, whereas Duplicator selects a vertex from the \textit{other} graph. Let $x_{\nu}$ denote the vertex selected from $X$ in round $\nu$, and $y_{\nu}$ that from $Y$ in round $\nu$, for $1 \leqslant \nu \leqslant k$.

Duplicator wins the game if for each $s, t \in \{1, \ldots, k\}$, the following hold:
\begin{itemize}
\item $x_s = x_t$ if and only if $y_s$ = $y_t$;
\item $x_s \sim x_t$ if and only if $y_s \sim y_t$.
\end{itemize}

Otherwise, Spoiler wins.\\


The following theorem is a well known corollary of the famous Ehrenfeucht-Fra\"{\i}ss\'{e} theorem about the equivalence of a logical indistinguishability and the existence of a winning strategy of Duplicator (see, e.g., \cite{uspekhi} for the complete proof of this corollary; the statement of Ehrenfeucht-Fra\"{\i}ss\'{e} theorem can be also found, e.g., in~\cite[Lemma 10.3]{luczak}~and~\cite[Theorem 3.18]{libkin}).

\textcolor{black}{Hereafter, for random graphs $X$, $Y$, we write $X\stackrel{d}=Y$ meaning that $X$ and $Y$ are identically distributed}.\\

\begin{theorem}
$G(n,p)$ obeys FO$_k$ 0-1 law if and only if
\begin{gather}\label{ehrenfeucht-formula}
    \lim_{n,m \rightarrow \infty} \p \Bigl(\mbox{Duplicator has a winning strategy in } \EHR(X, Y, k)\Bigr) = 1,
\end{gather}
where $X\stackrel{d}=G(n,p(n))$, $Y\stackrel{d}=G(m,p(m))$ are independent.\\
\label{ehrenfeuht}

\end{theorem}

{\color{black} A complete proof of this corollary can be found in~\cite{strangelogic,uspekhi}. }


\subsection{Subgraphs and extensions}
\begin{defin}
Consider a graph property $\mathcal{L}$.
\begin{itemize}
\item $\mathcal{L}$ is called \textit{increasing} if, for any pair of graphs \textcolor{black}{$G$, $G'$} with $V(G) = V(G')$ \textcolor{black}{and $E(G) \subset E(G')$}, $G \in \mathcal{L}$ implies $G' \in\mathcal{L}$. 
\item $\mathcal{L}$ is called \textit{decreasing} if, for any pair of graphs \textcolor{black}{$G$, $G'$} with $V(G) = V(G')$ \textcolor{black}{and $E(G) \supset E(G')$}, $G \in \mathcal{L}$ implies $G' \in\mathcal{L}$.
\end{itemize}
Increasing and decreasing properties are called {\it monotone}.\\
\end{defin}

In 1987, Bollob\'{a}s and Thomason proved that each increasing property has a {\it threshold function} \cite{bollobas}.\\

\begin{defin} Consider a monotone graph property $\mathcal{L}$.
\begin{itemize}
    \item $p_0=p_0(n)$ is called \textit{a threshold function} for an increasing property $\mathcal{L}$, if the following holds: 

\begin{center}
$p = o(p_0)$ $\Rightarrow$ $\lim\limits_{n\rightarrow\infty}\p (G(n,p) \in\mathcal{L}) =0$; $\quad p_0 = o(p)$ $\Rightarrow$ $\lim\limits_{n\to\infty}\p (G(n,p) \in \mathcal{L})=1$.
\end{center}

\item $p_0=p_0(n)$ is called \textit{a threshold function} for a decreasing property $\mathcal{L}$, if the following holds: 

\begin{center}
$p = o(p_0)$ $\Rightarrow$ $\lim\limits_{n\rightarrow\infty}\p (G(n,p) \in\mathcal{L}) =1$; $\quad p_0 = o(p)$ $\Rightarrow$ $\lim\limits_{n\to\infty}\p (G(n,p) \in \mathcal{L})=0$.
\end{center}

\end{itemize}

\end{defin}


Clearly, the property of containing a subgraph isomorphic to a given one is increasing and expressed in FO. For such a property, the threshold is known. Let $G$ be a graph with $v(G):=|V(G)|$ vertices and $e(G):=|E(G)|$ edges. Recall that the \textit{density of $G$} is $\rho(G) = \frac{e(G)}{v(G)}$.  The \textit{maximal density} of $G$ is $\rho^{\max} (G) = \max\limits_{H \subseteq G} \rho(H)$. Consider the property $\mathcal{L}_G$ of containing a subgraph \textcolor{black}{(not necessarily induced)} isomorphic to $G$. \\

\begin{theorem}[A. Ruci\'{n}ski, A. Vince, 1985, \cite{rucinski}] 
$n^{-1/\rho^{\max}(G)}$ is threshold for $\mathcal{L}_G$.\\ \label{theorem_no_dense_subgraphs}
\end{theorem}

\textcolor{black}{For $\alpha \in [\frac 35, \frac 35 + \varepsilon)$ and $\varepsilon>0$ small enough, Theorem \ref{theorem_no_dense_subgraphs} implies that $G(n, n^{-\alpha})$  contains a $4$-clique (as $3/5<2/3$) and does not contain a $5$-clique (as $3/5>1/2$) a.a.s. Moreover, any graph with four vertices appears in $G(n, n^{-\alpha})$ a.a.s. As noted in Section~\ref{intro}, we are looking for sentences from FO$_4$ that are not existential.\\
}

Consider graphs $G \supseteq H$, $\tilde G \supseteq \tilde H$ and positive integers $k\leqslant\ell$ such that
\begin{gather*}
V(H) = \set{x_1, \ldots, x_k},\,\, V(G) = \set{x_1, \ldots, x_{\ell}}, 
\\
V(\tilde H) = \set{\tilde x_1, \ldots, \tilde x_k},\,\, V(\tilde G) = \set{\tilde x_1, \ldots, \tilde x_{\ell}}.
\end{gather*}

\begin{defin}
The graph $\tilde G$ is a \textit{$(G,H)$-extension} of the graph $\tilde H$, if, for every $i\in[\ell]\setminus[k]$, $j\in[\ell]$,
$$\{x_i, x_j\} \in E(G) \setminus E(H) \Rightarrow \{\tilde x_i, \tilde x_j\} \in E(\tilde G) \setminus E(\tilde H).$$
If, for every $i\in[\ell]\setminus[k]$, $j\in[\ell]$,
$$\{x_i, x_j\} \in E(G) \setminus E(H) \Leftrightarrow \{\tilde x_i, \tilde x_j\} \in E(\tilde G) \setminus E(\tilde H),$$
then $\tilde G$ is a \textit{strict $(G, H)$-extension} of $\tilde H$, and the pairs $(G,H)$ and $(\tilde G, \tilde H)$ are \textit{isomorphic}.

Moreover, if $\tilde x_1, \ldots, \tilde x_k$ are not necessarily distinct (but $\tilde x_{k+1}, \ldots, \tilde x_{\ell}$ are all distinct), then $\tilde G$ is called a \textit{(strict) generalised $(G,H)$-extension of $\tilde H$}. \textcolor{black}{If a generalized extension is strict, the pairs $(G,H)$ and $(\tilde G, \tilde H)$ are \textit{generically isomorphic}}.\\
\end{defin}

\textcolor{black}{
Let us clarify the latter definitions. For both regular and generalised extensions, we 
do not forbid additional adjacencies in $E(\tilde G) \setminus E(\tilde H)$. In contrast, an extension is strict, if all pairs of vertices outside ${V(\tilde H)\choose 2}$ have prescribed adjacency relations. In all these definitions, edges inside $H$ and $\tilde H$ do not play a role. In particular, the ismorpism of pairs $(G, H)$ and $(\tilde G, \tilde H)$ does not imply that either $H$ and $\tilde H$ or $G$ and $\tilde G$ are isomorphic.}

The definitions above depend on the labelling of vertices of graphs $G$, $H$, $\tilde G$, $\tilde H$. Throughout the paper, the labelling follows from the context.\\

Fix a positive number $\alpha$. For graphs $H \subseteq G$, define
\begin{gather*}
V(G,H) = V(G) \setminus V(H),\quad E(G,H) = E(G) \setminus E(H),\\
v(G,H) = |V(G,H)|,\quad e(G,H) = |E(G,H)|,\\
f_{\alpha}(G,H) = v(G,H) - \alpha e(G,H).
\end{gather*}

Notice that for graphs $H$, $S$, $G$ such that $H \subseteq S \subseteq G$, we have the obvious equalities (used frequently in the paper):
$v(G,H) = v(G,S) + v(S,H)$ and $e(G,H) = e(G,S) + e(S,H)$, which imply
$f_{\alpha}(G,H) = f_{\alpha}(G,S) + f_{\alpha}(S,H)$ for all $\alpha>0$.\\

\begin{defin}
Let $H \subset G$ be a pair of graphs and let $\alpha$ be a positive number.
\begin{itemize}
    \item If for each graph $S$ such that $H \subset S \subseteq G$, the inequality $f_{\alpha}(S,H) > 0$ holds then the pair $(G,H)$ is called \textit{$\alpha$-safe}.
    
    \item If for each graph $S$ such that $H \subseteq S \subset G$ the inequality $f_{\alpha}(G,S) < 0$ holds then the pair $(G, H)$ is called \textit{$\alpha$-rigid}.
    
    \item If for each graph $S$ such that $H \subset S \subset G$ the inequality $f_{\alpha}(S,H) > 0$ holds, but also $f_{\alpha}(G,H) = 0$ then the pair $(G,H)$ is called \textit{$\alpha$-neutral}. Note that, if $(G,H)$ is \textit{$\alpha$-neutral}, then, for each graph $S$ such that $H \subset S \subset G$, the relation $f_{\alpha}(G,S) < 0$ holds.
\end{itemize}

If the pair $(G,H)$, $H \subset G$, is $\alpha$-safe, $\alpha$-rigid or $\alpha$-neutral, we call $G$ respectively an \textit{$\alpha$-safe, $\alpha$-rigid or $\alpha$-neutral extension} of $H$.
\end{defin}

\textcolor{black}{Observe that, in a standard way, from the union bound and the method of moments the following conclusion could be derived. It is likely that a fixed subgraph has an $\alpha$-safe extension, while it is not likely that it has an $\alpha$-rigid extension in $G(n, n^{-\alpha})$. The number of $\alpha$-neutral extensions has a non-trivial limit --- it converges to a Poisson random variable. The first two notions appear to be useful for proving 0-1 laws for irrational $\alpha$~\cite{spencershelah}, while $\alpha$-netrual extensions are used in the proof of Theorem~\ref{zhukovskiiless}.}


\textcolor{black}{Note} that if $H \subset G \subset G'$, $(G, H)$ is $\alpha$-safe ($\alpha$-rigid) and $(G', G)$ is $\alpha$-safe ($\alpha$-rigid) then $(G', H)$ is also $\alpha$-safe ($\alpha$-rigid).\\

Hereafter, for a graph $G$ and a subset of its vertices $V \subseteq V(G)$,  $G|_{V}$ denotes the subgraph of $G$ induced on $V$. The following claim is proven in~\cite{yu-rovikov}.

\begin{claim}[\cite{yu-rovikov}] \label{helpKTokralphasafe}
Let $\alpha>0$ be rational. Consider graphs $W \subset U \subseteq G$ such that the pair $(G, U)$ is $\alpha$-neutral, and the pair $(U, W)$ is $\alpha$-safe ($W$ can be empty; in this case, the pair $(U,W)$ is $\alpha$-safe whenever $\rho^{\max}(U)<1/\alpha$). Let there also be at least one edge between $V(G) \setminus V(U)$ and $V(U) \setminus V(W)$ in $G$. Then the pair $(G, W)$ is also $\alpha$-safe.  \\
\end{claim}

\textcolor{black}{We also need to introduce the notion of a generic extension (\cite{alonspencer}, \cite{strangelogic}). Roughly speaking, an extension is called $t$-generic if it has no rigid ``non-trivial'' extensions, where we call an extension trivial if it extends solely the smaller graph.}

\begin{defin}
Let $\alpha$ positive be fixed, and let $t$ be a positive integet. 

\begin{itemize}
    \item A single graph $G$ is called \textit{$t$-generic} if there are no $\alpha$-rigid extensions of $G$ with \textcolor{black}{at most} $t$ vertices.

    \item Let $H \subset G \subset \Gamma$ be a triple of nested subgraphs. The graph $G$ is called a \textit{$t$-generic extension} of $H$ if for each $K: G \subset K \subset \Gamma$ such that $(K, G)$ is $\alpha$-rigid and $v(K, G) \leqslant t$, there are no edges between $V(K, H)$ and $V(G, H)$. 
\end{itemize}
\end{defin}



\textcolor{black}{To clarify the latter definition, we give an example of a generic extension.} 
\begin{figure}[h!]
    \centering
    \includegraphics[scale=0.4,]{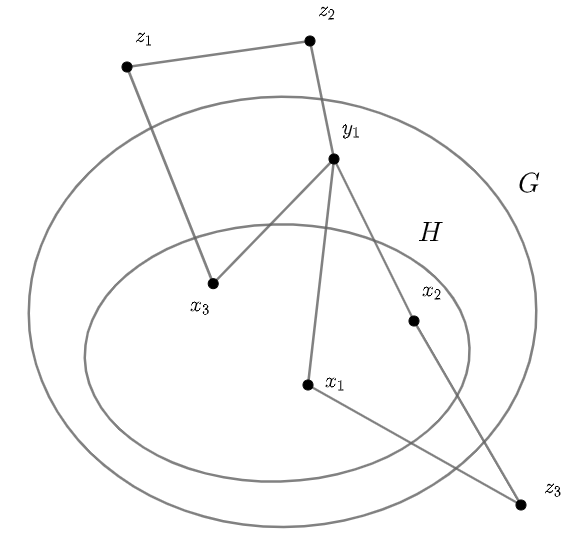}
    \caption{A generic extension}
    \label{fig:generic}
\end{figure}


\textcolor{black}{On Figure \ref{fig:generic}, $G$ is a $2$-generic extension of $H$ for $\alpha=3/5$, since the extension formed by $z_1$ and $z_2$ is $3/5$-safe and the extension formed by $z_3$ does not have edges to $V(G, H)$. Moreover, $G$ is not a $2$-generic extension of $H$ for $\alpha=1/3$ since the extension formed by $z_1$ and $z_2$ is $1/3$-rigid. However, $G$ is a $1$-generic extension for $\alpha=1/3$ (and, obviously, $3/5$).}\\

{\color{black} We will frequently use the following theorem,  that was implicitly proven in~\cite{spencershelah}. For convenince, we refer to~\cite[Lemma 10.4.6]{alonspencer}}.

\begin{theorem}[S. Shelah, J. Spencer, 1988, \cite{spencershelah}] \label{G,H,K,T} 
\textcolor{black}{Let a pair $(G, H)$, $H \subset G$, be $\alpha$-safe, $t\in\mathbb{N}$, $r = v(H)$.
Then a.a.s. each subgraph of $G(n, n^{-\alpha})$ on $r$ vertices has a strict $t$-generic $(G,H)$-extension. } 
\end{theorem}

\textcolor{black}{In particular, if $H$ is empty, we obtain that a.a.s. in $G(n, n^{-\alpha})$ there is a $t$-generic graph isomorphic to a given graph $G$ with $\rho(G) < 1/\alpha$.\\}

\section{New statements and constructions}
\label{sc:new_constructions}

In this section, for a pair of graphs $H\subset G$, we let $f(G,H):=f_{3/5}(G,H)$. We also denote $f(G)=v(G)-\frac{3}{5}e(G)$.

To prove Theorem~\ref{main_theorem}, we will show that there exists $\varepsilon>0$ such that, for $\alpha\in(\frac{3}{5},\frac{3}{5}+\varepsilon)$,  a.a.s. Duplicator has a winning strategy in the game on $G(n,n^{-\alpha})$, $G(m,m^{-\alpha})$ in 4 rounds. In this section, we define a finite set $\mathcal{G}$ of graphs that play a role of `bad' neighbourhoods of initial moves of Spoiler.



\subsection{Building the set $\G$}
The definition of $\G$ is loosely based on a definition of a similar set of graphs constructed in \cite{yu-rovikov}. 

Since we are only interested in isomorphism classes of graphs, it will be convenient to assume that all graphs in $\G$ contain a common singleton subgraph $H$ with one vertex: $V(H) = \set{z}$. 

Hereinafter we denote by $I(A)$ the indicator function that equals 1 if the event $A$ holds, and 0 otherwise.

\begin{defin}
Let $T$ be some graph such that $H \subset T$, $\gamma\in\mathbb{Z}_{>0}$. We call a pair of graphs $(K,T)$, $T \subset K$, \textit{$\gamma$-bad} if 
\begin{enumerate}[label=\alph*)]
\item $(K,T)$ is $\frac35$-rigid, 
\item $v(K,T) \leqslant \textcolor{black}{15}\cdot 2^{\gamma-1} I(\gamma\in\{1,2\})+\textcolor{black}{60}I(\gamma\geq 3)$\label{vkt},
\item $v(T) \geqslant 2$,
each vertex of $V(K) \setminus V(T)$ has a path to $V(T) \setminus \set{z}$ in the graph $K|_{V(K)\setminus \{z\}}$.\label{prop-3-bad-pairs}\\
\end{enumerate}
\end{defin}
Note that we bound $v(K,T)$ from above, since we are only interested in relatively small extensions when playing the Ehrenfeucht game in 4 rounds.\\

We define $\mathcal{G}$ as the union of sets of graphs $\mathcal{G}_i$, $i \in \mathbb{N}$, defined inductively as follows.\\

\noindent {\bf Definition (sets $\mathcal{G}_i$ and $\mathcal{G}$).}
Let $\G_0$ be a {\it minimum} set of graphs \textcolor{black}{$G$} with the common subgraph $H$ such that, if $\rho^{\max}(G)<\frac{5}{3}$, $v(G, H) \leqslant 10$, the pair $(G, H)$ is $\frac35$-neutral or $\frac35$-rigid, and $G|_{V(G) \setminus \set{z}}$ is a connected graph, then there is a graph isomorphic to $G$ in $\G_0$ (the image of $z$ is $z$ itself). $\G_0$ is finite since the set of all non-isomorphic pairs $(G,H)$ with $v(G,H)\leq 10$ is finite.

For $i=1,2,\ldots$, let $\mathcal{G}_i$ be a {\it minimum} set of graphs with the common subgraph $H$ and the following property. If $G \in \G_{i-1}$ and a graph $G' \supset G$ is such that the pair $(G', G)$ is $i$-bad and $\rho^{\max}(G')<\frac{5}{3}$, then there is a graph isomorphic to $G'$ in $\G_i$ (the image of $z$ is $z$ itself). Set $\G:=\bigsqcup_{i=0}^{\infty}\mathcal{G}_i$.

Without loss of generality, let us assume that any pair of graphs from $\G$ have $z$ as their only common vertex.\\

Those definitions are partially motivated by the fact that the maximal density of each graph in $\G$ should be less than $5/3$ (otherwise, such graphs do not appear as subgraphs in the considered random graphs).

Notice that, according to this definition, any graph of $\mathcal{G}$
belongs to one and only one set $\mathcal{G}_i$. For convenience, we will often write
$G\in \mathcal{G}$ (resp. $G\in \mathcal{G}_i$) to signify that a graph $G$ has an
isomorphic copy in $\mathcal{G}$ (resp. $\mathcal{G}_i$).\\


\textcolor{black}{Consider} a sequence $H\subset G_0\subset G_1\subset\ldots\subset G_5$ such that $v(G_0,H)\leq 10$, $f(G_0,H)\leq 0$, $(G_i,G_{i-1})$ are $i$-bad for all $i\in\{1,2,3,4,5\}$. We have $f_{3/5}(G_i,G_{i-1}) < 0$ for all $i$, which implies $e(G_i,G_{i-1})\geq
\frac 53 v(G_i,G_{i-1})+\frac 13$ (since the functions $e(\cdot)$ and $v(\cdot)$ have integer
values). We obtain

$$
e(G_5)=e(G_0)+\sum_{i=1}^5 e(G_i,G_{i-1})\geq \frac{5}{3}v(G_0,H)+\sum_{i=1}^5 \left[\frac{5}{3}v(G_i,G_{i-1})+\frac{1}{3}\right]=
$$
$$
\frac{5}{3}v(G_5,H)+\frac{5}{3}=\frac{5}{3}v(G_5)-\frac{5}{3}+\frac{5}{3}=\frac{5}{3}v(G_5).
$$
Therefore, $f(G_5)=v(G_5) - \frac35 e(G_5)\leqslant 0$, which implies
$\rho^{\max}(G_5) \geq \frac{e(G_5)}{v(G_5)} = \frac 53$. According to the definition of the sets of graphs $\mathcal{G}_i$, this yields $\mathcal{G}_i = \varnothing$ for all $i\geq 5$.

Thus, $\mathcal{G}$ is the (finite) union of the $\mathcal{G}_i$ for $i\leq 4$. As a consequence, for each $G\in \mathcal{G}$, \textcolor{black}{there is an integer} $i \in \{0,1,2,3,4\}$
and a sequence of graphs $H \subset G_0 \subset \dots \subset G_i = G$ such that
$G_j \in \mathcal{G}_j$, for $0 \leq j \leq i$.
We will use the subgraphs $G_j$ in the properties 1-7 of this subsection.\\  

Let us list several properties of $\mathcal{G}$. 

\gpr\label{Prop:f_from_above} $f(G)\leq 1 - \frac i5$ for $G \in \mathcal{G}_i$, $i\in\{0,1,2,3,4\}$.

\begin{proof}
Since the pair $(G_0,H)$ is $\frac35$-neutral or $\frac35$-rigid we have $f(G_0,H)\leq 0$.
It implies $f(G_0)\leq f(H) = 1$ as required. Moreover, for $i\in \{1,2,3,4\}$ we have
$f(G_i)-f(G_{i-1})=f(G_i,G_{i-1}) < 0$, since $(G_i,G_{i-1})$ is $\frac35$-rigid. Note that
$f(G_i,G_{i-1})$ is of the form $x/5$ for an integer $x$.
Thus, we get $f(G_i)-f(G_{i-1}) = f(G_i,G_{i-1}) \leq -1/5$. Therefore, we inductively
obtain $f(G_i) \le 1-\frac i5$.
\end{proof} 

\gpr For $G \in \G$, we have $\rho^{\max}(G) < \frac53$, since this is a necessary condition for a graph $G$ to belong to any set $\mathcal{G}_i$.

\gpr \label{graphs are small} For $G \in \G$, we have $v(G)\leq \textcolor{black}{176}$ vertices.

\begin{proof}
We have $G \in \mathcal{G}_i$ with a sequence of graphs $H \subset G_0
\subset \dots \subset G_i = G$, with each $G_j \in \mathcal{G}_j$.
We also have $v(G_i)= v(G_i,G_{i-1})+\dots + v(G_0,H) + v(H)$ which is at most {\color{black} $1 + 10 + 15
+ 30 + 2 \cdot 60=176$} by condition \ref{vkt}.
\end{proof}


%
%
%

\gpr \label{graphs are dense} For any $G \in \G\setminus\mathcal{G}_0$, the pair $(G, H)$ is $\frac35$-rigid. 
\begin{proof}
If $G \in \G\setminus\mathcal{G}_0$, then there exists a chain of graphs $H \subset G_0 \subset G_1 \ldots \subset G_i = G$ such that $(G_0, H)$ is $\frac 35$-rigid or $\frac 35$-neutral and $(G_j, G_{j-1})$ is $j$-bad, for $j=1,\ldots,i$. Consider an arbitrary graph $S$ such that $H \subseteq S \subset G$. It is sufficient to show that $f(G,S) < 0$. Let $S_j = G_j \cap S$. Notice that, since $G = G|_{V(G_i) \cup V(S)}$, we have 
\begin{multline*}
    f(G,S) = f(G|_{V(G_0) \cup V(S)}, S) + 
    \\ 
    + f(G|_{V(G_1) \cup V(S)}, G|_{V(G_0) \cup V(S)}) + \ldots + f(G|_{V(G_i) \cup V(S)}, G|_{V(G_{i-1}) \cup V(S)}).
\end{multline*}
Note that the terms on the right side of the equality are non-positive. Indeed, $f(G|_{V(G_0) \cup V(S)}, S) \leqslant f(G_0, S_0) \leqslant 0$ (the latter inequality follows from the fact that $(G_0, H)$ is $\frac35$-rigid or $\frac35$-neutral). Furthermore, 
\begin{equation}
f(G|_{V(G_j) \cup V(S)}, G|_{V(G_{j-1}) \cup V(S)}) \leqslant f(G_j, G|_{V(G_{j-1})\cup V(S_j)}) \leqslant 0,
\label{eq:ineq_two_f}
\end{equation}
since $(G_j, G_{j-1})$ is $\frac35$-rigid and for all $1 \leqslant j \leqslant i$. 

Assume that all the terms are equal to zero. Since, for $j\geq 1$, $f(G|_{V(G_j) \cup V(S)}, G|_{V(G_{j-1}) \cup V(S)}) = 0$, due to (\ref{eq:ineq_two_f}), we obtain $f(G_j, G|_{V(G_{j-1})\cup V(S_j)}) = 0$. \textcolor{black}{Since $V(G_{j-1}) \subseteq V(G_{j-1})\cup V(S_j) \subseteq V(G_j)$ and $(G_{j}, G_{j-1})$ is $\frac 35$-rigid,} the latter equality implies $G_j = S_j$. Therefore, $G_j \setminus G_{j-1} \subseteq S$ for all $j \geqslant 1$. Thus, $V(G) = V(S) \sqcup (V(G_0) \setminus V(S_0))$. Moreover, if $S_0 \neq H$ then $f(G|_{V(G_0) \cup V(S)}, S) < 0$, since $(G_0, H)$ is $\frac35$-rigid or $\frac35$-neutral. Finally, if $S_0 = H$ then $f(G|_{V(G_0) \cup V(S)}, S) = f(G_0, H) - \frac35 t$ where $t$ is the number of edges between $V(G_0) \setminus V(H)$ and $V(S) \setminus V(H)$. Obviously, $t > 0$ since there is at least one edge between $V(G_0) \setminus V(H)$ and $V(G_1) \setminus V(G_0)$ (by Property \ref{prop-3-bad-pairs} from the definition of a bad pair). We obtain that $f(G|_{V(G_0) \cup V(S)}, S)< f(G_0, H) \leqslant 0$, which contradicts the assumption.
\end{proof}

\gpr \label{2_joined} For each $G \in \G$, the graph $G|_{V(G) \setminus \{z\}}$ is connected.
\begin{proof}
Let us prove the statement by induction on $i$. For $i=0$ and $G \in \mathcal{G}_i$, the property is trivial from the definition of $\mathcal{G}_0$.

Let us perform the induction step. If $G=G_i \in \mathcal{G}_i$, there exists $G_{i-1} \in \mathcal{G}_{i-1}$ such that $(G_i, G_{i-1})$ is an $i$-bad pair.  
By definition, each vertex of $V(G_i) \setminus V(G_{i-1})$ has a path to $V(G_{i-1}) \setminus \{z\}$. Moreover, $G_{i-1}|_{V(G_{i-1}) \setminus \{z\}}$ is connected. This completes the induction step.
\end{proof}

Let us also prove the following property.

{\color{black}
\gpr \label{two-bad-subgraphs-cant-intersect-kjtj} Let $G_1$, $G_2\in\G$. Let $K$ be the graph obtained from $G_1 \cup G_2$ (recall that $V(G_1) \cap V(G_2) = z$) by extending it with a $\frac{3}{5}$-rigid pair $(\tilde K, \tilde T)$ with $v(\tilde K, \tilde T) \leqslant 5$ such that each vertex of $V(K)\setminus (V(G_1) \cup V(G_2))$ has a path to both $V(G_1) \setminus \{z\}$ and $V(G_2) \setminus \{z\}$ in $K|_{V(K) \setminus \{z\}}$. Then either $\rho^{\max}(K)\geqslant 5/3$ or $K \in \mathcal{G}$.

\begin{proof}
Without loss of generality, let $G_1 \in \mathcal{G}_i$, $G_2 \in \mathcal{G}_j$, $i \leqslant j$. Note that $i \leqslant 1$, otherwise
$$f(G)=f(G_1)+f(G_2) - 1 + f(\tilde K, \tilde T) \leq 1 - \frac{i_1}{5} + 1 - \frac{i_2}{5} - 1 - \frac{1}{5} = \frac{4}{5}-\frac{i_1+i_2}{5} \leqslant 0,$$
which implies $\rho(G) \geqslant \frac 53$. Denote $T = K|_{V(G_1) \cup V(G_2)}$. Let us prove that $(K, G_2)$ is an $(i+1)$-bad pair. 

We first prove that $(K, G_2)$ is $\frac 35$-rigid. Consider a graph $S$ such that $G_2 \subseteq S \subset K$. We should verify that $f(K, S) < 0$. Consider three cases: 
$V(G_1) \subseteq V(S)$, $\{z\} \subsetneq V(G_1) \cap V(S) \subsetneq V(G)$ or $V(G_1) \cap V(S) = \{z\}$.

In the first case, we have $T \subseteq S \subset K$, which implies that $f(K, S)<0$ since $(K, T)$ is $\frac 35$-rigid.

Consider the second case. We have the identity 
$$f(K, S) = f(K, G_1 \cup S) + f(G_1\cup S, S).$$
Since $v(G_1 \cup S, S) = v(G_1, S \cap G_1)$ and $e(G_1, S \cap G_1) \leqslant e(G_1\cup S, S)$, we obtain $f(G_1 \cup S, S) \leqslant f(G_1, S \cap G_1)$. Moreover, $f(G_1, S \cap G_1) < 0$ since $(G_1, H)$ is $\frac 35$-neutral or $\frac 35$-rigid and $V(S \cap G_1)$ strictly includes $H$. Thus, $f(G_1\cup S, S) < 0$.
Finally, $f(K, G_1\cup S) \leqslant 0$ since $T \subseteq G_1 \cup S$. We therefore obtain that $f(K, S) = f(K, G_1 \cup S) + f(G_1\cup S, S) < 0$.

Consider the third case. We have $f(K, S) = f(K, G_1 \cup S) + f(G_1 \cup S, S)$. As in the previous case, both terms are non-positive. Moreover, if $V(S)\cup V(G_1)=V(K)$ then $f(G_1 \cup S, S) < 0$ since there is at least one edge between $V(G_1) \setminus \{z\}$ and $V(S) \setminus \{z\}$. If, on the other hand, $V(S)\cup V(G_1)\subsetneq V(K)$ then $f(K, G_1 \cup S) < 0$ since $(K, T)$ is $\frac 35$-rigid. In any case, $f(K, S) < 0$. 

Thus, $(K, G_2)$ is $\frac 35$-rigid.

Let us prove that each vertex of $V(K, G_2)$ has a path to $V(G_2) \setminus \{z\}$ in $K|_{V(K) \setminus \{z\}}$. It is true for the vertices of $V(K, T)$ by the definition of $K$. Moreover, each vertex of $V(K, T)$ has a path to at least one vertex of $V(G_1) \setminus \{z\}$. By Property \ref{2_joined}, $G_1$ remains connected after removing $z$, so there is a path from each vertex of $V(G_1) \setminus \{z\}$ to a vertex of $V(K,T)$ and, therefore, to $V(G_2) \setminus \{z\}$, as needed. 

Finally, $v(K, G_2) \leqslant v(G_1, H) + 5$, from which it follows that $v(K, G_2)$ is $(i+1)$-bad (since $i \leqslant 2$). Since $i \leqslant j$, we obtain that either $\rho^{\max}(K)\geqslant 5/3$ or $K \in \mathcal{G}$.
\end{proof}

During the proof of Theorem \ref{main_theorem}, we will require Property \ref{two-bad-subgraphs-cant-intersect-kjtj} for several specific cases. We list these cases below. 

\begin{corollary}
Let $G_1, G_2 \in \G$.
\begin{enumerate}
\item Let $G$ be obtained by joining $G_1\cup G_2$ with a new vertex $v$ that has neighbours $u_1 \in V(G_1)$ and $u_2 \in V(G_2)$, $u_1,u_2\neq z$. Then either $G \in \mathcal{G}$ or $\rho^{\max}(G) \geqslant 5/3$.


\item Let $G$ be the following graph. We start from the union of graphs $G_1$ and $G_2$ merged by $z$. We then add to $G$ distinct vertices $z_1,z_2,z_3$ and $t_1,t_2,t_3$ and build two $3/5$-neutral extensions: 1) an extension of $G_1$ on $z_1, z_2, z_3$ such that each $z_i$ has a path to $V(G_1) \setminus \{z\}$, 2) an extension of $G_2$ on $t_1, t_2, t_3$ such that each $t_i$ has a path to $V(G_2) \setminus \{z\}$. Finally, we fix a non-empty set of indices $J\subseteq \{1,2,3\}$ and identify $z_i$ with $t_i$ for all $i \in J$. Then either $G \in \G$ or $\rho^{\max}(G) \geqslant 5/3$.

\end{enumerate}
\end{corollary}
\begin{proof}
Item 1 is an obvious corollary of Property \ref{two-bad-subgraphs-cant-intersect-kjtj}. 

To prove Item 2, we need to vertify that $(G, G_1\cup G_2)$ is $3/5$-rigid. Consider an arbitraty subhgraph $S$ of $G$ such that  $S \supseteq G_1\cup G_2$. Let $S_1$ be a subgraph of $G$ induced on $V(S)$ and $z_1, z_2, z_3$ (recall that some of $t_i$ are merged with $z_i$). Since the extension of $G_1$ is $\frac 35$-neutral, we have $$f(G, S) = f(G, S_1) + f(S_1, S) \leqslant f(G, S_1).$$
Moreover, since the extension of $G_2$ is $\frac 35$-neutral and $v(G, S_1) < 3$, we have $f(G, S_1) < 0$. Thus, $(G, G_1\cup G_2)$ is $3/5$-rigid. The other conditions in the statement of Property \ref{two-bad-subgraphs-cant-intersect-kjtj} are trivial. Thus, either $G \in \G$ or $\rho^{\max}(G) \geqslant 5/3$.
\end{proof}
}

\subsection{Occurrence of graphs of $\G$ in an arbitrary graph}

Here we exploit the set $\G$ constructed in the previous section. Recall that $H$ is the singleton graph with $V(H)=\{z\}$. 

Consider a graph $\Gamma$ and a vertex $u\in V(\Gamma)$.

\begin{defin}
We call a graph $R\subset\Gamma$ containing $u$ {\it $\G$-maximal} if, for each $G' \in \G$, there is no $(G', H)$-extension $R' \subseteq \Gamma$ of $\Gamma|_u$ such that $R$ is a proper subgraph of $R'$. 

We call an induced subgraph $U \subset \Gamma$ \textit{$u$-bad} if it is a $\G$-maximal strict $(G, H)$-extension of $\Gamma|_u$ for some $G \in \G$; $G$ is called {\it the origin of $U$}.\\
\end{defin}

{\color{black} Let us recall that every $G \in \G$ has at most \textcolor{black}{$176$} vertices.} \\

\begin{claim} \label{all_bad_dont_intersect}
Assume $\Gamma$ does not contain any subgraph $G$ with $\rho(G)\geq\frac53$ \textcolor{black}{and $v(G)\leq 176 \cdot 2 - 2 = 350$.} Then, for each vertex $u \in V(\Gamma)$, any two $u$-bad subgraphs of $\Gamma$ only intersect in $u$.\\
\end{claim}

\begin{proof}
Assume the contrary. Let $U,W \subseteq \Gamma$ be $u$-bad subgraphs with $v(U \cap W) > 1$. Let $G_1\in\mathcal{G}_{i_1}$, $G_2\in\mathcal{G}_{i_2}$ be the origins of $U$ and $W$ respectively. Let $\min\{i_1,i_2\}=i_1$. Let $\varphi: G_1\cup G_2\to U\cup W$ be an endomorphism such that $\varphi(z)=u$, $\varphi|_{V(G_1)}$ is an isomorphism between $G_1$ and $U$, $\varphi|_{V(G_2)}$ is an isomorphism between $G_2$ and $W$. Let $T$ be the induced subgraph of $G_1$ that corresponds to $V(U\cap W)$ in $G_1$, that is, $V(T) = \left[\varphi|_{V(G_1)}\right]^{-1}(V(U\cap W))$. Note that the pair $(G_1,T)$ is $\frac{3}{5}$-rigid. Indeed, for each $S$ such that $T \subseteq S \subset G_1$ we have $f(G_1, S) < 0$ since $H \subsetneq T$.\\


If $i_1=0$, then $(U\cup W,W)$ is $\frac{3}{5}$-rigid (since $(G_1,T)$ is $\frac{3}{5}$-rigid) and $v(U\cup W,W)\leq 9$. If $i_1=1$, then $(U\cup W,W)$ is $\frac{3}{5}$-rigid and $v(U\cup W,W)\leq \textcolor{black}{24}$. Moreover, each vertex of $V(U \cup W) \setminus V(W)$ has a path to $V(W) \setminus \{u\}$ by Property \ref{2_joined}. In both cases, either $\rho^{\max}(U\cup W)\geq\frac{5}{3}$ (which contradicts the fact that $\Gamma$ does not contain subgraphs with $\rho(G)\geq\frac53$ and $v(G)\leq \textcolor{black}{350}$, \textcolor{black}{since $v(U \cup W) \leqslant 176 \cdot 2 - 2 =350$}) \textcolor{black}{or $U \cup W$ has a copy in $\mathcal{G}$}, and we get a contradiction with the assumption that $W$ is $u$-bad. Indeed, if $i_1=0$, then the pair $(U\cup W,W)$ is 1-bad, and, if $i_1=1$, then the pair $(U\cup W,W)$ is 2-bad, which is enough since $i_2\geq i_1=1$.\\


Finally, assume that $i_1\geq 2$.
Let $H\subset G_1^0\subset G_1^1\subset\ldots\subset G_1^{i_1}=G_1$ be such that
\begin{itemize}
\item $G_1^0\in\mathcal{G}_0$, i.e. $(G_1^0, H)$ is $\frac35$-neutral or $\frac35$-rigid and $v(G_1^0, H) \leqslant 10$; 
\item $G_1^i\in\mathcal{G}_i$, i.e. $(G_1^i,G_1^{i-1})$ are $i$-bad for all $i\in\{1,\ldots,i_1\}$.
\end{itemize}
For every $i\in\{0,1,\ldots,i_1\}$, let $U_i$ be the image of $G_1^i$ under $\varphi$.
Find the minimum $i\in\{0,1,\ldots,i_1\}$ such that $G_1^i\setminus T$ is non-empty. Notice that 
{\color{black}$$
v(U_i\cup W,W)\leq v(G_1^i)-v(G_1^{i-1})\leq 10I(i=0)+15I(i=1)+30I(i=2)+60I(i\geq 3),
$$}
where $G_1^{-1}=H$. Also, if $i\geq 1$, then the pair $(U_i\cup W,W)$ is $\frac{3}{5}$-rigid by the definition of $i$-bad pairs. Since $i_2\geq i_1\geq i$ and each vertex of $V(U_i \cup W)\setminus V(W)$ has a path to $V(W)\setminus\{u\}$ (by Property~\ref{2_joined}  and the fact that $G_1^i\in\mathcal{G}$), we get a contradiction with the assumption that $W$ is $u$-bad.

It remains to consider the case $i=0$. If $T$ shares at least 2 vertices with $G_1^0$, then the pair $(U_0\cup W,W)$ is $\frac{3}{5}$-rigid and each vertex of $V(U_0 \cup W)\setminus V(W)$ has a path to $V(W) \setminus \{u\}$ (by Property~\ref{2_joined} and the fact that $G_1^0\in\mathcal{G}$). It contradicts the assumption that $W$ is $u$-bad.

Assume that $T\cap G_1^0=\{z\}$. Also, assume that $G_1^1\setminus(G_1^0\cup T)$ is non-empty. Then $(U_1,W\cap U_1)$ is $\frac{3}{5}$-rigid. 

\begin{itemize}
    \item If $W\cap U_1$ has at least $2$ vertices, then it has a vertex other than $u$ adjacent to at least one vertex of $U_1\setminus W$ (otherwise, we get a contradiction with Property~\ref{2_joined} since $G^1_1\in\mathcal{G}$). Moreover, $v(U_1\cup W,W)\leq \textcolor{black}{24}$ and the pair $(U_1\cup W,W)$ is $\frac{3}{5}$-rigid. Moreover, each vertex of $V(U_1 \cup W) \setminus V(W)$ has a path to $V(W) \setminus \{u\}$. This contradicts the assumption that $G_2$ is $u$-bad. 
    
    \item If $V(W\cap U_1)=\{u\}$, then either $V(W\cup U)=V(U_1\cup W)$ or both pairs $(U_1\cup W,W)$ and $(W\cup U,U_1\cup W)$ are $\frac{3}{5}$-rigid. 
    In the former case, $U_1\setminus W$ has neighbours in $W$ other than $u$ (due to Property~\ref{2_joined} and the fact that $G_1\in\mathcal{G}$). Since $v(U_1\cup W,W)\leq \textcolor{black}{24}$, the pair $(U_1\cup W,W)$ is $\frac{3}{5}$-rigid, and each vertex of $V(U_1\cup W) \setminus V(W)$ has a path to $V(W) \setminus \{u\}$, this contradicts the assumption that $W$ is $u$-bad. 
    
    In the latter case, by Property~\ref{Prop:f_from_above},
\begin{align*}
e(W\cup U) & =e(W\cup U,W)+e(W) \\
& =e(W\cup U,U_1\cup W)+e(U_1\cup W,W)-\frac{5}{3}f(W)+\frac{5}{3}v(W)\\
& \geq\frac{5}{3}v(W\cup U,U_1\cup W)+\frac{1}{3}+\frac{5}{3}v(U_1\cup W,W)+\frac{1}{3}-\frac{5}{3}+\frac{i_2}{3}+\frac{5}{3}v(W)\\
& =\frac{5}{3}v(W\cup U)+\frac{i_2-3}{3}.
\end{align*}
Therefore, $i_1=i_2=2$ (otherwise, $\rho(U\cup W)\geq\frac{5}{3}$, which contradicts the fact that $\Gamma$ does not contain subgraphs with $\rho(G)\geq\frac53$ and $v(G)\leq \textcolor{black}{350}$). But then $(U\cup W, W)$ is $3/5$-rigid, $v(U\cup W,W)\leq \textcolor{black}{54}$, and each vertex of $V(U \cup W) \setminus V(W)$ has a path to $V(W) \setminus \{u\}$. We get a contradiction with the assumption that $W$ is $u$-bad.
\end{itemize}

Finally, we assume that $V(G_1^1\setminus G_1^0)$ is a subset of $T$. Then $G_1^0\setminus\{z\}$ has edges to $G_1^1\setminus G_1^0$ by Property~\ref{2_joined} and the fact that $G_1^1\in\mathcal{G}$. Then $(U_0\cup W,W)$ is $\frac{3}{5}$-rigid, $v(U_0\cup W,W)\leq 10$, and each vertex of $V(U_0 \cup W) \setminus V(W)$ has a path to $V(W) \setminus \{u\}$. We get a contradiction with the assumption that $W$ is $u$-bad.\\

\end{proof}

Enumerate all graphs in $\G$: $\G = \set{G_1, \ldots, G_{\kappa}}$. Let $V(G_i) = \set{z=z_1^i, z_2^i, \ldots, z_{v(G_i)}^i}$, $i\in[\kappa]$. 

\begin{defin}
For every $i\in[\kappa]$ and every strict $(G_i,H)$-extension $U$ of $u$ in $\Gamma$, we call an ordering $(u=u_1,\ldots,u_{v(G_i)})$ of the vertices of $U$ {\it canonical}, if $f:V(G_i)\to V(U)$ sending $z_j^i$ to $u_j$, $j\in[v(G_i)]$, is an isomorphism between $G_i$ and $U$.\\
\end{defin}

\begin{defin}
For $i\in[\kappa]$, let $\mathcal{U}_i(u)$ be the set of all $u$-bad $(G_i, H)$-extensions of the vertex $u$ in $\Gamma$.  Let the \textit{$0$-neighbourhood of the vertex} $u$ be the set $$U_0(u) = V(\Gamma)\setminus\left(\bigcup\limits_{i=1}^{\kappa} \bigcup\limits_{U \in \mathcal{U}_i(u)} V(U) \right).$$
\end{defin}

\section{Three main lemmas}
\label{sc:3_lemmata}

Here we formulate three lemmas that help Duplicator win the Ehrenfeucht game. Let us first introduce some auxiliary notations.

Let $\Gamma$ be an arbitrary graph. Let $v_1, \ldots, v_r$ be some vertices of $\Gamma$. For a tuple $(\alpha_1, \ldots, \alpha_r) \in \set{0,1}^r$, put 
$$\Nei\left(v_1^{\alpha_1}, \ldots, v_r^{\alpha_r}\right) = \set{u \in V(\Gamma) \mid \forall \, k \,(v_k \sim u \leftrightarrow \alpha_k = 1)} \setminus \set{v_1, \ldots, v_r}.$$ 
Also, for convenience (if possible), in place of $v_i^0$ we write $\lnot v_i$, in place of $v_i^1$ -- just $v_i$. For example, $$\Nei(\lnot v_1, v_2, \lnot v_3) = \set{u \in V(\Gamma) \mid u \not\sim v_1, u \sim v_2, u \not\sim v_3, u \not\in \set{v_1,v_2,v_3}}.$$
Put $$\delta\left(v_1^{\alpha_1}, \ldots, v_r^{\alpha_r}\right) = I\left(\left|\Nei\left(v_1^{\alpha_1}, \ldots, v_r^{\alpha_r}\right)\right| \geqslant 1\right).$$
Moreover, let $U$ be a subgraph of $\Gamma$ such that $v_i \not\in V(U)$, $i\in[r]$. Put
\begin{gather*}
\Nei^U\left(v_1^{\alpha_1}, \ldots, v_m^{\alpha_m}\right) = \Nei\left(v_1^{\alpha_1}, \ldots, v_m^{\alpha_m}\right) \setminus V(U);\\
\delta^U\left(v_1^{\alpha_1}, \ldots, v_r^{\alpha_r}\right) = I\left(\left|\Nei\left(v_1^{\alpha_1}, \ldots, v_r^{\alpha_r}\right) \setminus V(U)\right| \geqslant 1\right).
\end{gather*}



\subsection{Lemma 1}\label{sc:l1}
Lemma \ref{lemma1} (stated below in this section) enables Duplicator to make the first move in the game $\EHR(X\stackrel{d}=G(n,n^{-\alpha}),Y\stackrel{d}=G(m,m^{-\alpha}),4)$. Let $x_1$ be the first move by Spoiler in $V(X)$. Our goal is to find a vertex $y_1 \in V(Y)$ such that, for each $x_1$-bad subgraph $U_X \subset X$, there exists a $y_1$-bad subgraph in $Y$ isomorphic to $U_X$ that has similar small $3/5$-neutral extensions in $Y$ to those of $U_X$ in $X$, and vice versa.\\

More specifically, consider the two $3/5$-neutral pairs $(K_1, T_1)$, $(K_2, T_2)$ presented on Figure~\ref{fig:three_pairs}. 
Let $\nu_j = v(T_j)$, $j\in\{1,2\}$. Note that $2 \leqslant \nu_j \leqslant 3$ and $v(K_j, T_j) = 3$. We will be interested in generalised $(K_j, T_j)$-extensions. It is worth mentioning that 
strict generalised $(K_j,T_j)$-extensions are $3/5$-neutral for $j = 1,2$. Let $K$ be a generalised $(K_j,T_j)$-extension of certain vertices $v_1,\ldots,v_{\nu_j}$ (some of them may be equal). Let $v_{\nu_j+1},\ldots,v_{\nu_j+3}$ be the vertices of $K$ other than $v_1,\ldots,v_{\nu_j}$. For any other $v_1^*,\ldots,v_{\nu_j}^*$ (such that $v_{i_1}^*=v_{i_2}^*$ if and only if $v_{i_1}=v_{i_2}$) and its generalised $(K_j,T_j)$-extension $K^*$, we call the ordering of its vertices $(v_1^*,\ldots,v_{\nu_j+3}^*)$ {\it canonical}, if $f:V(K)\to V(\tilde K^*)$ sending $v_i$ to $v_i^*$, $i\in[\nu_j+3]$, is an isomorphism of $(V(K),E(K)\setminus E(K|_{v_1,\ldots,v_{\nu_j}}))$ and $(V(\tilde K^*),E(\tilde K^*)\setminus E(\tilde K^*|_{v_1^*,\ldots,v_{\nu_j}^*}))$.


\begin{figure}[h!]
    \centering
    \includegraphics[scale=0.4,]{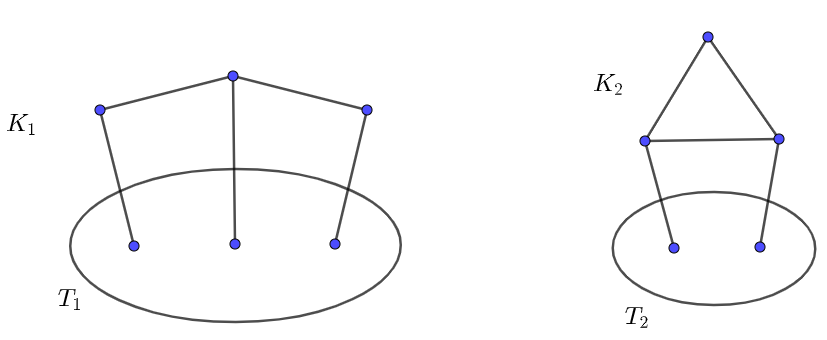}
    \caption{Pairs $(K_1, T_1)$ and $(K_2, T_2)$}
    \label{fig:three_pairs}
\end{figure}


For a vertex $u$ of a graph $\Gamma$, we are going to define its profile as the set of all possible ways (under some depth restrictions) to extend $u$-bad subgraphs of $\Gamma$ via $(K_j,T_j)$. The Duplicator's strategy in the first round provided by Lemma~\ref{lemma1} is to choose $y_1$ with the same profile in $Y$ as $x_1$ has in $X$.\\

Fix $j\in\{1,2\}$ and a subgraph $U\subset\Gamma$. Let $v_1, \ldots, v_{\nu_j}$ be some (not necessarily distinct) vertices in $U$.
Put $\zeta_j^{U}(v_1, \ldots, v_{\nu_j}) = 1$ if in the subgraph of $\Gamma$ induced on $V(\Gamma) \setminus \left(V(U) \setminus \set{v_1, \ldots, v_{\nu_j}}\right)$ there is a strict generalised $(K_j, T_j)$-extension $K$ of  $v_1, \ldots, v_{\nu_j}$. Otherwise, put $\zeta_j^{U}(v_1, \ldots, v_{\nu_j}) = 0$. 

Assume that $\zeta_j^{U}(v_1, \ldots, v_{\nu_j}) = 1$ and consider a respective strict generalised $(K_j, T_j)$-extension $K$ of $v_1,\ldots, v_{\nu_j}$.
Let $(v_1,\ldots,v_{\nu_j},v_{\nu_j + 1},v_{\nu_j + 2},v_{\nu_j + 3})$ be a canonical order of the vertices of $K$. Consider the function
$$
\theta^{U}(K): \set{(s_1,s_2,s_3) : 1 \leqslant s_1,s_2,s_3 \leqslant \nu_j + 3, \{s_1, s_2, s_3\} \not\subseteq \{1, \ldots, \nu_j\}} \rightarrow \set{0,1}
$$
defined as $\theta^{U}(K)(s_1,s_2,s_3)=\zeta_1^{K \cup U}(v_{s_1}, v_{s_2}, v_{s_3})$.
Roughly, $\theta^{U}(K)$ is a function that takes $3$ indices $s_1$, $s_2$, $s_3$ (at least one of which is greater than $\nu_j$) and returns $1$ if and only if there is a strict generalised $(K_1,T_1)$-extension of $v_{s_1}$, $v_{s_2}$, $v_{s_3}$  that does not use any vertices of $K \cup U$. In other words, $\theta^{U}(K)$ describes all ``second-level extensions'' of $K$ (since $K$ is an extension itself). 
We call $\theta^{U}(K)$ a \textit{specification} of $K$. 

Let $u_1$ be a vertex of $\Gamma$ and $U$ be a $u_1$-bad subgraph of $\Gamma$ (isomorphic to $G_i$). Consider a canonical order $(u_1, \ldots, u_{v(G_i)})$ of the vertices of $U$. Let $T_j(U)$ be the vector 
$$\left(\set{\theta^{U}(K) \mid \text{$K$ is a strict generalised $(K_j, T_j)$-extension of }u_{a_1}, \ldots, u_{a_{\nu_j}}}\right)_{1 \leqslant a_1, \ldots, a_{\nu_j} \leqslant v(G_i)}.$$ 
The indexing here are by all tuples of indices of length $\nu_j$ of vertices (not necessarily distinct) in $U$. This vector describes all specifications of strict generalised 
$(K_j, T_j)$-extensions of all subgraphs of $U$. 

Finally, for $i\in[\kappa]$, define the {\it $i$-profile} of a vertex $u$ of $\Gamma$ as $\J^i(u) = \set{T(U) \mid U \in \mathcal{U}_i(u)}$, where $
T(U) = (T_1(U), T_2(U))$.\\

\begin{lemma} \label{lemma1}
There exists $\varepsilon > 0$ such that for each  $\alpha \in \left(\frac35, \frac35 + \varepsilon\right)$ a.a.s $(X,Y)$ satisfies the following condition. For each vertex $x_1 \in V(X)$ there exists $y_1 \in V(Y)$ such that, for all $1 \leqslant i \leqslant\kappa$,
$$
\J^i(x_1) = \J^i(y_1).
$$
\end{lemma}

\begin{proof}
The informal idea of the proof is as follows. For an arbitrary vertex $x_1 \in V(X)$ we need to construct a graph $Z$ with a vertex $z_1 \in Z$ such that for all $i$: $1 \leqslant i \leqslant \kappa$ the equality $\J^i(x_1) = \J^i(z_1)$ holds. We then find a \textcolor{black}{$t$-generic} copy $\hat Z$ of $Z$ in $Y$ for some large $t$ and some $\alpha$ close to $3/5$. Finally, we choose $y_1$ to be the image of $z_1$ under the isomorphism between $Z$ and $\hat Z$.\\

Let $C>0$ be a large enough constant. For each positive integer $n$ let $\mathcal{X}_n \subseteq \Omega_n$ be the set of all graphs on $[n]$ without subgraphs $G$ with $\rho^{\max}(G) \geqslant 5/3$ and $v(G) \leqslant C$. By Theorem \ref{theorem_no_dense_subgraphs}, we have $\p(X \in \mathcal{X}_n) \rightarrow 1$ for each $\alpha > 3/5$. Consider an arbitrary $A \in \mathcal{X}_n$ and a vertex \textcolor{black}{$x_1 \in A$}. By Claim \ref{all_bad_dont_intersect}, any two $x_1$-bad subgraphs of $A$ only intersect in $x_1$. We now introduce an algorithm to build a graph $Z = Z(A, x_1)$ with $z_1 \in V(Z)$ such that $\rho^{\max}(Z) < 5/3$ and the equalities $\J^i(x_1) = \J^i(z_1)$ hold for $1 \leqslant i \leqslant \kappa$.

\begin{enumerate}
    \item Let $z_1$ be a \textcolor{black}{\it new} vertex. 
    \item For every $1 \leqslant i \leqslant \kappa$, for each $T \in \J^i(x_1)$ and the respective $x_1$-bad subgraph $U\subset A$ (isomorphic to $G_i$) with a canonical order of its vertices $u_1,\ldots,u_{v(G_i)}$: \label{main_1}
    \begin{enumerate}
        \item Select arbitrary {\it new} vertices $z_2=z_2(T), \ldots, z_{v(G_i)}=z_{v(G_i)}(T)$ and draw edges between them so that the subgraph $\tilde U$ of $Z$ induced on $z_1, z_2, \ldots, z_{v(G_i)}$ is isomorphic to $G_i$ and the ordering $(z_1,z_2,\ldots, z_{v(G_i)})$ is canonical.\label{build_subgraph_1}
        \item For all $j = 1,2$,
        for all $1 \leqslant a_1, \ldots, a_{\nu_j} \leqslant v(G_i)$ such that $\zeta^U_{j}(u_{a_1}, \ldots, u_{a_{\nu_j}})=1$, and, for every element of the $(a_1,\ldots,a_{\nu_j})$-th coordinate of $T_j(u)$
        and a respective strict generalised $(K_j,T_j)$-extension $K$ of $u_{a_1},\ldots,u_{a_{\nu_j}}$:
        \begin{enumerate}
        	\item Construct a strict generalised $(K_j, T_j)$-extension $\tilde K$ of $z_{a_1}, \ldots, z_{a_{\nu_j}}$. Let $(w_1,\ldots,w_p)$ be a canonical order of vertices of $\tilde K$. \label{build_type_1}
        	\item For all triplets $(s_1,s_2,s_3)$ such that $1 \leqslant s_1,s_2,s_3 \leqslant p$, $\set{w_{s_1}, w_{s_2}, w_{s_3}} \not\subseteq V(\tilde U)$: if $\theta^{U}( K)(s_1,s_2,s_3) = 1$, construct a strict generalised $(K_1,T_1)$-extension of $w_{s_1},w_{s_2},w_{s_3}$. \label{build_type_2}
        \end{enumerate}
    \end{enumerate}
\end{enumerate}

Let $Z_0$ be the subgraph of $Z$ induced on $z_1$ and the vertices selected at steps \ref{build_subgraph_1} of the algorithm. Let us list several properties of $Z$. Notice, first, that the number of steps~\ref{build_subgraph_1} that the algorithm performs is bounded from above by an absolute constant and, therefore, the number of vertices in $Z$ is bounded by an absolute constant.\\

\begin{claim}
The inequality $\rho^{\max}(Z) < 5/3$ holds.\\
\label{cl:Z_sparse}
\end{claim}

\begin{proof}
Consider a subgraph $\tilde U$ of $Z$ constructed at step \ref{build_subgraph_1}, isomorphic to $U\subset A$. All such subgraphs of $A$ are members of $\mathcal{U}_i(x_1)$. These subgraphs can only intersect in $x_1$, by Claim \ref{all_bad_dont_intersect}. The union of all $U$ (we take one at every step \ref{build_subgraph_1}) denoted by $A_0$ is isomorphic to $Z_0$. Notice that the number of steps~\ref{build_subgraph_1} that the algorithm performs is bounded from above by an absolute constant. Therefore, we may choose $C$ so large that $v(Z_0)=v(A_0)\leq C$. Therefore, $\rho^{\max}(Z_0) = \rho^{\max}(A_0)< 5/3$. Moreover, $Z$ is obtained from $Z_0$ by applying consequent $3/5$-neutral extensions of $Z_0$. By Claim \ref{helpKTokralphasafe}, $\rho^{\max}(Z) < 5/3$.\\
\end{proof}


\begin{claim}
The set $\mathcal{U}_i(z_1)$ is precisely the set of all $(G_i, H)$-extensions of $z_1$ in $Z$ constructed at steps \ref{build_subgraph_1}.\\
\label{cl:same_bad_set}
\end{claim}

\begin{proof}
Let $U \subseteq Z$ be a subgraph constructed  at step \ref{build_subgraph_1} of the algorithm. 
By Claim~\ref{all_bad_dont_intersect} and Claim~\ref{cl:Z_sparse}, it is sufficient to show that $U$ is $\G$-maximal. 

\textcolor{black}{Assume the contrary. If $U$ is not $\mathcal{G}$-maximal then there is $W \supset U$ such that $W$ has a copy in $\mathcal{G}$. Let $U'$ be the subgraph of $Z$ that contains $U$ and all the first-level and second-level extensions of $U$ constructed at steps \ref{build_type_1} and \ref{build_type_2} of the algorithm. By Property \ref{2_joined} of the set $\mathcal{G}$, $W$ cannot contain vertices of other subgraphs constructed at step \ref{build_subgraph_1} of the algorithm, except for $U$. This implies that $W \subset U'$. Note that $U'$ is obtained from $U$ by adding $3/5$-neutral (and thus $\alpha$-safe for all $\alpha > 3/5$) extensions. This implies that $(U', U)$ is $\alpha$-safe for all $\alpha>3/5$. Therefore, $(W, U)$ is also $\alpha$-safe for all $\alpha>3/5$ and hence not $3/5$-rigid, with $v(U) \geqslant 2$. If $W \in \mathcal{G} \setminus \mathcal{G}_0$, we obtain a contradiction with Property \ref{graphs are dense} of $\mathcal{G}$. If, on the other hand, $W \in \mathcal{G}_0$, we obtain a contradiction with the fact that $(W, Z_0)$ is $3/5$-neutral.}
\end{proof}

\begin{claim} \label{x=z}
For each $i$ the equalities $\J^i(x_1) = \J^i(z_1)$ hold.\\
\end{claim}
\begin{proof}
The equalities directly follow from the definition of the algorithm. For each $T \in \J^i(x_1)$ the algorithm constructs a $z_1$-bad extension $\tilde U$ for which $T(\tilde U) = T$. Therefore, $\J^i(x_1) \subseteq \J^i(z_1)$. By Claim~\ref{cl:same_bad_set}, there are no other $z_1$-bad extensions in $Z$, which implies the reverse inclusion.
\end{proof}

Denote $\mathcal{Z}_n = \set{Z(A, x_1) \mid A \in \mathcal{X}_n, x_1 \in V(A)}$. Denote $\mathcal{Z} = \bigcup\limits_{n \in \N} \mathcal{Z}_n$ and $$\varepsilon =\frac{1}{\sup\limits_{Z \in \mathcal{Z}}\rho^{\max}(Z)} - \frac{3}{5}.$$ 
Since all $Z$ have bounded sizes, we get that $\mathcal{Z}$ is finite. Since all graphs in $\mathcal{Z}$ have a maximal density less than $5/3$, we have $\varepsilon > 0$.



Consider $\mathcal{Y}_m \subseteq \Omega_m$ --- the set of all graphs $B$ on $[m]$ such that for each $Z \in \mathcal{Z}$, in $B$ there exists a \textcolor{black}{$175$-generic} induced subgraph isomorphic to $Z$. By Theorem \ref{G,H,K,T}, for every $\alpha \in \left(3/5, 3/5 + \varepsilon\right)$, we have $\p(Y \in \mathcal{Y}_m) \rightarrow 1$. Consider arbitrary graphs $Z \in \mathcal{Z}$ and $B \in \mathcal{Y}_m$. Find in $B$ an induced subgraph $\hat Z$ isomorphic to $Z$ with the above maximality condition. Let $\psi:V(Z) \rightarrow V(\hat Z)$ be an isomorphism between the graphs. Set $y_1 = \psi(z_1)$. \textcolor{black}{First of all, there are no $y_1$-bad subgraphs in $B$ other than the copies of $z_1$-bad subgraphs in $Z$, since a $y_1$-bad subgraph $\tilde U$ would form an extension of $B|_{\set{y_1}}$ with $v(\tilde U, \tilde U|_{\set{x_1}}) \leqslant 176-1=175$, which contradicts the fact that $\hat Z$ is $175$-generic}. Clearly, by Property \ref{graphs are dense} of the set $\mathcal{G}$ and Claim~\ref{all_bad_dont_intersect}, $\J^i(z_1) = \J^i(y_1)$ for all $i$. 
From Claim \ref{x=z}, it follows that, for each pair $(A, B)$: $A \in \mathcal{X}_n$, $B \in \mathcal{Y}_n$, and for each $x_1 \in V(A)$ there exists a vertex $y_1 \in V(B)$ for which the equalities $\mathcal{J}^i(x_1) = \mathcal{J}^i(y_1)$ hold, which concludes the proof. \end{proof}

\subsection{Lemmas 2 and 3}
Lemmas \ref{lemma2} and \ref{lemma3} help Duplicator make a move in round $2$ of $EHR(X,Y,4)$ in case Spoiler chooses a vertex, say, $x_2 \in V(X)$ that is not contained in any $x_1$-bad subgraph. In this case, the strategy should be to pick a vertex $y_2 \in V(Y)$ such that the subgraphs induced on $x_1, x_2$ and $y_1, y_2$ have similar dense extensions in $X$, $Y$. 

Below, we first define the set of `dense' extensions of interest.\\

Define $(K^*, T^*)$ as a pair of graphs with $v(T^*) = 2$, $e(T^*) = 0$, $v(K^*) = 3$, the only vertex in $V(K^*, T^*)$ being adjacent to each vertex from $V(T^*)$ (a ``tick'' extension).

Let $U$ be an induced subgraph of $\Gamma$. Let
\begin{itemize}
\item $W_0 = V(U)$;
\item $W_{i+1} = W_{i} \cup \bigl\{v \in V(\Gamma) : \left|\set{w \in W_{i} : w \sim v }\right| \geqslant 2\bigr\}$, $i \in \N$;
\item $W = \bigcup\limits_{i=0}^{\infty} W_i$.
\end{itemize}
Let us call $\Gamma|_{W}$ the \textit{$(K^*, T^*)$-neighbourhood} of $U$ in $\Gamma$. Finally, the subgraph of $\Gamma$ induced on ${\bigcup\limits_{i=0}^{k} W_i}$ is called the \textit{$k$-order $(K^*, T^*)$-neighbourhood} of $U$ in $\Gamma$.


In Lemma \ref{lemma2}, we focus on the case when $x_2$ is distant from $x_1$ and all $x_1$-bad subgraphs. Later, in Lemma \ref{lemma3}, we will analyze all the remaining $x_2$.\\

\begin{lemma} \label{lemma2}
There exists $\varepsilon > 0$ such that for each $\alpha \in \left(\frac35, \frac35 + \varepsilon\right)$ a.a.s. $(X,Y)$ satisfies the following condition.
For each pair of vertices $x_1, x_2 \in V(X)$ and for a vertex $y_1 \in V(Y)$ such that
\begin{enumerate} 
\item $\J^i(x_1) = \J^i(y_1)$ for all $i \in \kappa$, \label{l2_begin_cond}
\item $x_2 \in U_0(x_1)$, \label{l2_U0}
\item $x_2$ does not have neighbours in $x_1$-bad subgraphs of $X$ (except for, possibly, $x_1$),
\item for any $x_1$-bad subgraph $U\subset X$, for any $j\in\{1,2\}$, for any tuple $\tilde T$ of $\nu_j$ vertices from $U$, there is no generalised $(K_j, T_j)$-extension of $\tilde T$ that contains $x_2$ ($(K_j, T_j)$ is defined in Section~\ref{sc:l1}), \label{l2_cond_no_kt_extension} \label{l2_end_cond}
\end{enumerate}
there exists a vertex $y_2 \in V(Y)$ such that
\begin{enumerate}
\item $y_2 \in U_0(y_1)$,
\item $x_1 \sim x_2 \Leftrightarrow y_1 \sim y_2$,
\item $y_2$ does not have neighbours in $y_1$-bad subgraphs of $Y$ (except for, possibly, $y_1$),
\item for any $y_1$-bad subgraph $U\subset Y$, for any $j\in\{1,2\}$, for any tuple $\tilde T$ of $\nu_j$ vertices from $U$, there is no generalised $(K_j, T_j)$-extension of $\tilde T$ that contains $y_2$,
\item the pairs $(W_X, X|_{\set{x_1,x_2}})$ and $(W_Y, Y|_{\set{y_1,y_2}})$ are isomorphic, where $W_X$, $W_Y$ are the $3$-order $(K^*, T^*)$-neighbourhoods of $X|_{\set{x_1,x_2}}$, $Y|_{\set{y_1,y_2}}$ in $X$, $Y$,
\item \label{lemma2-end-conclusion} there exists an isomorphism $\varphi:V(W_X)\to V(W_Y)$ between $W_X$ and $W_Y$ such that $\varphi(x_1)=y_1$, $\varphi(x_2)=y_1$ and, for every $j\in\{1,2\}$, for each tuple $(w_1,\ldots,w_{\nu_j})$ of vertices from $W_X$ the equality $\zeta_j^{W_X}(w_1,\ldots,w_{\nu_j}) = \zeta_j^{W_Y}(\varphi(w_1),\ldots,\varphi(w_{\nu_j}))$ holds. \\
\end{enumerate}

\end{lemma}

\begin{proof}

For each positive integer $n$ let $\mathcal{X}_n \subseteq \Omega_n$ be the set of graphs on $n$ vertices without subgraphs $G$ with $\rho^{\max}(G) \geqslant 5/3$ and $v(G) \leqslant 12$. By Theorem \ref{G,H,K,T}, we have $\p(X \in \mathcal{X}_n) \rightarrow 1$ for each $\alpha > 3/5$. Consider an arbitrary $A\in\mathcal{X}_n$ and vertices $x_1$, $x_2$ such that the if-conditions \ref{l2_U0}-\ref{l2_cond_no_kt_extension} hold. Let $W_A$ be the $3$-order $(K^*, T^*)$-neighbourhood of $A|_{\set{x_1,x_2}}$. 
Let us list some properties of $W_A$.\\

\begin{claim}
$v(W_A) \leqslant 11$.\\
\end{claim}

\begin{proof}
Assume the contrary. Then there is a graph $\tilde W_A$ with $v(\tilde W_A) = 12$ and $A|_{\set{x_1, x_2}} \subset \tilde W_A \subseteq W_A$ such that
\begin{align*}
f(\tilde W_A) & \textcolor{black}{\leqslant} 2 + f(\tilde W_A, A|_{\set{x_1, x_2}})\\
 &< 2 + v(\tilde W_A, A|_{\set{x_1, x_2}})-\frac{6}{5}v(\tilde W_A, A|_{\set{x_1, x_2}})=2 - \frac{v(\tilde W_A) - 2}{5}=0.
\end{align*}

We obtain that $\rho(\tilde W_A) > 5/3$, which contradicts the fact that $A$ does not contain subgraphs with density at least $5/3$ on no more than 12 vertices.\\ \end{proof}

\begin{claim}
The pair $(W_A,A|_{\set{x_1}})$ is $3/5$-safe.\\
\label{cl:w_A_safe}
\end{claim}

\begin{proof}
Assume the contrary. If $(W_A, A|_{\set{x_1}})$ is not $3/5$-safe then there is a subgraph $\tilde U \subseteq W_A$ such that $(\tilde U, A|_{\set{x_1}})$ is $3/5$-rigid or $3/5$-neutral. Since $v(\tilde U) \leqslant 11$, $\tilde U$ has a copy in $\G$ and, therefore, is a subgraph of an $x_1$-bad subgraph in \textcolor{black}{$W_A$}. Let $U$ be an $x_1$-bad subgraph in $W_A$ that contains $\tilde U$. $U$ is obviously a subgraph of an $x_1$-bad subgraph in $A$. Thus, $x_2$ does not have neighbours in $U$ (except for, possibly, $x_1$). 

Let $W_i$ be the set of vertices of the $i$-order $(K^*, T^*)$-neighbourhood of $A|_{\set{x_1,x_2}}$. Also put $W_0 = \set{x_1,x_2}$, $W_i' = W_i \setminus W_{i-1}$ and $U_i = V(U) \cap W'_i$. \textcolor{black}{Note that $x_2 \not\in V(U)$ due to the if-condition 2, hence $V(U) = \{x_1\}\sqcup U_1 \sqcup U_2 \sqcup U_3$.}

First of all, since $x_2$ does not have neighbours in $V(U)\setminus\{x_1\}$, we have $U_1 = \varnothing$. Furthermore, let $w \in U_2$. Then either $w \in \mathcal{N}(x_1,s)$ for some $s \in W_1$ or $w \in \mathcal{N}(s,t)$ for some $s,t \in W_1$. In both cases, there exists $s \in \mathcal{N}(w, x_1)\cap W_1$. Since $w \in V(U)$, $x_1 \in V(U)$ and $U$ is $x_1$-bad in $W_X$, we get $s \in V(U)$, which contradicts the fact that $U_1 = \varnothing$. Therefore, $U_2$ is also empty. Finally, since $U \in \G$, $x_1$ has a neighbour in $U$ and, consequently, $U_3$. Let $v_1 \in U_3$, $v_1 \sim x_1$. Since $v_1 \in U_3$, there is $v_2 \in W_2'$ s.t. $v_2 \sim v_1$. If $v_2 \sim x_1$ then $v_2 \in U$, which contradicts the fact that $U_2 = \varnothing$. If $v_2$ has at least two neighbours $s, t \in W_1'$ 
then the pair $(W_A|_{\set{x_1, v_1, v_2, s,t,x_2}}, W_A|_{\set{x_1, v_1}})$ is $1$-bad, which contradicts the fact that $U$ is $x_1$-bad in $W_A$. Thus, $v_2 \sim x_2$ and there is $s \in W_1'$ s.t. $s \sim v_2$. Then the pair $(W_X|_{\set{x_1,v_1,v_2,s,x_2}}, W_X|_{\set{x_1, v_1}})$ is either $1$-bad or isomorphic to $(K_2,T_2)$, which contradicts the if-condition \ref{l2_cond_no_kt_extension}, since $v_1$ is a vertex of an $x_1$-bad subgraph in $X$ and since $x_2$ does not belong to this subgraph (due to the if-condition 2). Thus, $U_3 = \varnothing$, and we obtain a contradiction. \end{proof}

Let us build a graph $Z = Z(A, x_1, x_2)$. Initially, we set $Z = W_A$. Then, for each $j$ and each tuple $(w_1,\ldots,w_{\nu_j})$ of vertices from $W_A$ such that $\zeta_j^{W_A}(w_1,\ldots,w_{\nu_j}) = 1$, we construct a strict generalised $(K_j, T_j)$-extension of $(w_1,\ldots,w_{\nu_j})$ and add those vertices and edges to $Z$. By Claim \ref{helpKTokralphasafe} and Claim~\ref{cl:w_A_safe}, we obtain that $(Z,Z|_{\{x_1\}})$ is $3/5$-safe.

Denote 
$$    \mathcal{Z}_n = \set{Z(A, x_1, x_2) \mid A \in \mathcal{X}_n, \, x_1, x_2 \in V(A), \, \text{the if-conditions \ref{l2_U0}-\ref{l2_end_cond} hold}};\quad
    \mathcal{Z} = \bigcup\limits_{n \in \N} \mathcal{Z}_n. 
    $$

Clearly, $v(Z(A, x_1, x_2))$ is bounded by an absolute constant. Denote
$$\varepsilon = \inf\limits_{Z \in \mathcal{Z}}\frac{v(Z, Z|_{\{x_1\}})}{e(Z, Z|_{\{x_1\}})} - \frac{3}{5}.$$
Obviously, $\varepsilon > 0$.

Let $\mathcal{Y}_m \subseteq \Omega_m$ be the set of all graphs $B$ on $[m]$ such that for each $Z \in \mathcal{Z}$ and $y_1 \in V(B)$ there exists a strict a \textcolor{black}{$175$-generic} $(Z, Z|_{\{x_1\}})$-extension of $B|_{\set{y_1}}$ in $B$. By Theorem \ref{G,H,K,T}, we have $\p(Y \in \mathcal{Y}_m) \rightarrow 1$. 

Let $y_1 \in V(B)$ be a vertex (if exists) such that $\J^i(x_1) = \J^i(y_1)$ for all $i\in[\kappa]$. By the definition of $\mathcal{Y}_m$, in $B$, there is a strict $175$-generic $(Z, Z|_{\set{x_1}})$-extension $\hat Z$ of $B|_{\set{y_1}}$. We now choose $y_2$ to be the image of $x_2$ under an isomorphism between $Z$ and $\hat Z$ that maps $x_1$ to $y_1$. 

It is now easy to verify all the properties of $y_2$. The only non-trivial property is that $y_2$ does not have neighbours in any $y_1$-bad subgraph. Assume the contrary: let there be a $y_1$-bad subgraph $U\subset B$ and \textcolor{black}{$w \in V(U) \setminus \set{y_1}$} such that $w \sim y_2$. Then $U$ cannot be completely inside $\hat Z$, since there are no $x_1$-bad subgraphs in $\hat Z$ by Claim~\ref{cl:w_A_safe}. 
\textcolor{black}{Let $\tilde U =  B|_{V(U) \cup V(\hat Z)}$}. The pair $(\tilde U, \hat Z)$ is $3/5$-rigid: \textcolor{black}{indeed, if $U$ has a copy in $\mathcal{G} \setminus \mathcal{G}_0$ then $(\tilde U, \hat Z)$ is $3/5$-rigid by Property~\ref{graphs are dense} of $\mathcal{G}$; if $U$ has a copy in $\mathcal{G}_0$ then $(\tilde U, \hat Z)$ is $3/5$-rigid since $(U, U|_{\set{y_1}})$ is $3/5$-neutral or $3/5$-rigid (by the definition of $\mathcal{G}_0$) and there is an extra edge between $y_2$ and $w$.
Note that $v(\tilde U, \hat Z) \textcolor{black}{\leqslant v(U) - 1} \leqslant 175$ and there is an edge between $y_2$ and $w$, which contradicts the fact that $\hat Z$ is a $175$-generic extension of $B|_{\set{y_1}}$}.\\ \end{proof}


\begin{lemma} \label{lemma3}
There exists $\varepsilon > 0$ such that for each $\alpha \in \left(\frac35, \frac35 + \varepsilon\right)$ a.a.s. $(X,Y)$ satisfies the following condition.
For each pair of vertices $x_1, x_2 \in V(X)$ and for a vertex $y_1 \in V(Y)$ such that 
\begin{enumerate} 
\item $\J^i(x_1) = \J^i(y_1)$ for all $i\in[\kappa]$, \label{l3_begin_cond}
\item $x_2 \in U^0(x_1)$, 
\item $x_2 \not \sim x_1$,
\item there is only one $x_1$-bad subgraph $U_X$ that contains a vertex adjacent to $x_2$. There are no more vertices in $U_X$ adjacent to $x_2$.
\item for any $x_1$-bad subgraph $U\subset X$, for any $j\in\{1,2\}$, for any tuple $\tilde T$ of $\nu_j$ vertices from $U$, there is no generalised $(K_j, T_j)$-extension of $\tilde T$ that contains $x_2$, \label{l3_end_cond}
\end{enumerate}
there exists a vertex $y_2 \in V(Y)$ such that
\begin{enumerate}
\item $y_2 \in U^0(y_1)$,
\item $y_2\not\sim y_1$,
\item there is only one $y_1$-bad subgraph $U_Y$ that contains a vertex adjacent to $y_2$. There are no more vertices in $U_Y$ adjacent to $y_2$. Moreover, $T(U_X) = T(U_Y)$, and the vertex adjacent to $y_2$ is an image of the vertex in $U_X$ adjacent to $x_2$ under $\psi$, where $\psi$ is an isomorphism between $U_X$ and $U_Y$ such that $\psi(x_1) = y_1$ and $\psi$ preserves the canonical order, \label{then-cond-2-lemma-3}
\item for any $y_1$-bad subgraph $U\subset Y$, for any $j\in\{1,2\}$, for any tuple $\tilde T$ of $\nu_j$ vertices from $U$, there is no generalised $(K_j, T_j)$-extension of $\tilde T$ that contains $y_2$,
\item the pairs $(W_X, X|_{V(U_X) \cup \set{x_2}})$ and $(W_Y, Y|_{V(U_Y) \cup \set{y_2}})$ are isomorphic, where $W_X$, $W_Y$ are the $1$-order $(K^*, T^*)$-neighbourhoods of $X|_{V(U_X) \cup \set{x_2}}$, $Y|_{V(U_Y) \cup \set{y_2}}$ in $X$, $Y$,
\item there exists an isomorphism $\varphi:V(W_X)\to V(W_Y)$ between $W_X$ and $W_Y$ such that $\varphi|_{U_X} = \psi$ ($\psi$ is defined in Property \ref{then-cond-2-lemma-3}) and, for every $j\in\{1,2\}$, for each tuple $(w_1,\ldots,w_{\nu_j})$ of vertices from $W_X$, the equality $\zeta_j^{W_X}(w_1,\ldots,w_{\nu_j}) = \zeta_j^{W_Y}(\varphi(w_1),\ldots,\varphi(w_{\nu_j}))$ holds. 
\\
\end{enumerate}

\end{lemma}

\begin{remark}
From the conditions of Lemma \ref{lemma3} it follows that $|V(W_X) \setminus (V(U_X) \cup \set{x_2})| \leqslant 1$, which is shown below in the proof of Lemma~\ref{lemma3}.
\end{remark}

\begin{proof}
The idea of the proof is similar to that of Lemma \ref{lemma2}. We build an auxiliary graph $Z\supset U_X$ such that $(Z, U_X)$ is $3/5$-safe and the existence of a generic $(Z,U_X)$-extension of $U_Y$ in $Y$ entails the existence of such a vertex $y_2$.\\

Consider an arbitrary $A \in \Omega_n$ and vertices $x_1$, $x_2$ such that the if-conditions \ref{l2_U0}-\ref{l3_end_cond} hold. Let $W_A$ be the $1$-order $(K^*, T^*)$-neighbourhood of $A|_{\set{x_1,x_2}}$. 

Note that $|V(W_A) \setminus (V(U_A) \cup \set{x_2})| \leqslant 1$: there is no more than one vertex outside $U_A$ adjacent to $x_2$ and one of the vertices in $U_A$. Otherwise $x_2$ is contained in an extension of $U_A$ generically isomorphic to $(K_1, T_1)$, which contradicts the if-condition \ref{l3_end_cond}. Moreover, there is no vertex outside $U_A$ adjacent to at least two vertices in $U_A$, since $U_A$ is $x_1$-bad.

Let us build a graph $Z = Z(A, x_1, x_2)$. Initially $Z = W_A$. Then, for each $j$ and each tuple $w_1,\ldots,w_{\nu_j}$ of vertices from $V(W_A)$ such that $\zeta_j^{W_A}(w_1,\ldots,w_{\nu_j}) = 1$, we construct a strict generalised $(K_j, T_j)$-extension of $w_1,\ldots,w_{\nu_j}$ and add those vertices and edges to $Z$. Note that the pair $(Z, U_A)$ is $3/5$-safe: indeed, $(W_A, U_A)$ is $3/5$-safe; $(Z,U_A)$ is $3/5$-safe by Claim \ref{helpKTokralphasafe}. Denote 
$$    
\mathcal{Z}_n = \set{Z(A, x_1, x_2) \mid A \in \Omega_n, \, x_1, x_2 \in V(A), \, \text{the if-conditions 2-5 hold}};\quad
    \mathcal{Z} = \bigcup\limits_{n \in \N} \mathcal{Z}_n. 
$$
Let $\mathcal{Y}_m \subseteq \Omega_m$ be the set of all graphs $B$ on $[m]$ that satisfy the following condition. For each $Z = Z(A,x_1,x_2) \in \mathcal{Z}$ and each $y_1 \in V(B)$ and $U_B \subset B$ such that the if-conditions of Lemma \ref{lemma3} hold, there exists a strict \textcolor{black}{$175$-generic} $(Z, U_A)$-extension of $U_B$ in $B$. By Theorem \ref{G,H,K,T}, we have $\p(B \in \mathcal{Y}_m) \rightarrow 1$ (in the same way as in the proof of Lemma~\ref{lemma2}). Now, let $B \in \mathcal{Y}_m$. Let also $y_1 \in V(B)$ be a vertex such that $\J^i(x_1) = \J^i(y_1)$ for all $i\in[\kappa]$. Then there exists a $y_1$-bad subgraph $U_B$ in $Y$ such that $T(U_A) = T(U_B)$. By the definition of $\mathcal{Y}_m$, in $B$ there is a strict \textcolor{black}{$175$-generic} $(Z,U_A)$-extension of $U_B$. It is easy to verify that $y_2$ satisfies all the conditions of Lemma \ref{lemma3}. \end{proof}


\newcounter{mycasescounter}
\setcounter{mycasescounter}{0}
\def\i{\refstepcounter{mycasescounter}\textbf{Case \arabic{mycasescounter}.}\ }

\newcounter{pun}[mycasescounter]
\def\pu{\refstepcounter{pun}{\bf Subcase \arabic{mycasescounter}\alph{pun}.}\ }

\newcounter{subpun}[pun]
\def\subpu{\refstepcounter{subpun}{\bf Subcase \arabic{mycasescounter}\alph{pun}(\roman{subpun}).}\ }

\def\ob{\setcounter{mycasescounter}{0}}
\renewcommand{\thepun}{\arabic{mycasescounter}\alph{pun}}

\section{Proof of the main theorem}
\label{sc:th_proof}

In this section, we prove Theorem \ref{main_theorem}. We need to show that $3/5$ is not a limit point of the $4$-spectrum. Note that, as mentioned in Section \ref{intro}, for each $k$, the $k$-spetrum can only have limit points ``from above''.  

By Theorem \ref{ehrenfeuht}, it is sufficient to show that there exists $\varepsilon > 0$ such that for each $\alpha \in \left(3/5, 3/5 + \varepsilon\right)$ Duplicator a.a.s.\ has a winning strategy in the game $\EHR(X, Y, 4)$, where $X\stackrel{d}= G(n,n^{-\alpha})$ and $Y\stackrel{d}= G(m,m^{-\alpha})$ are independent random graphs.\\

Let $\varepsilon$ be so small that statements of Lemmas \ref{lemma1}, \ref{lemma2} and \ref{lemma3} hold for this value. Consider $\alpha \in \left(3/5, 3/5 + \varepsilon\right)$.

Consider the set $\mathcal{P}$ of all graph pairs $(A,B)$ satisfying the following conditions:

\begin{enumerate}
\item \label{begin_conds_XY} \label{lemma1holds} \textcolor{black}{For each $x_1 \in V(A)$ all $x_1$-bad subgraphs only intersect in $x_1$, the same holds for $B$.} For each $x_1 \in V(A)$ there exists $y_1 \in V(B)$ such that $\J^i(x_1) = \J^i(y_1)$ for all $i\in[\kappa]$, and vice versa. 

\item\label{lemma2holds} For each pair of vertices $x_1$, $x_2 \in V(A)$ and each vertex $y_1 \in V(B)$ such that Conditions \ref{l2_begin_cond}-\ref{l2_end_cond} of Lemma \ref{lemma2} hold, there exists $y_2 \in V(B)$ that satisfies the conclusion of Lemma \ref{lemma2}, and vice versa.

\item\label{lemma3holds} For each pair of vertices $x_1$, $x_2 \in V(A)$ and each vertex $y_1 \in V(B)$ such that Conditions \ref{l3_begin_cond}-\ref{l3_end_cond}  of Lemma \ref{lemma3} hold, there exists $y_2 \in V(B)$ that satisfies the conclusion of Lemma \ref{lemma3}, and vice versa.

\item\label{end_conds_XY} For each subgraph $S \subseteq A$ with $v(S) \leqslant 3$ and each $3/5$-safe pair $(G, H)$ such that $v(G, H) \leqslant 5$, in $A$ there exists a strict \textcolor{black}{$1$}-generic $(G,H)$-extension of $S$, and the same holds for $B$. 

\item\label{no-dense-subgraphs} \textcolor{black}{$X$ and $Y$ do not contain any subgraphs $G$ with $\rho(G) \geqslant 5/3$ and $v(G) \leqslant 176 \cdot 2 = 352$.}
\end{enumerate}

From \textcolor{black}{Claim \ref{all_bad_dont_intersect}}, Lemmas \ref{lemma1}, \ref{lemma2}, \ref{lemma3}, and Theorems \ref{G,H,K,T} \textcolor{black}{and \ref{theorem_no_dense_subgraphs}}, it follows that $\p((X,Y) \in\mathcal{P}) \rightarrow 1$.
Let us prove that Duplicator has a winning strategy on a pair $(A, B) \in \mathcal{P}$.\\

Without loss of generality we may assume that in the first round Spoiler chooses a vertex $x_1 \in V(A)$. Duplicator chooses a vertex $y_1 \in V(B)$ such that \begin{equation}
\J^i(x_1) = \J^i(y_1)\quad\text{ for all }i.
\label{equality_from_P1}
\end{equation}

From now on, during the proof we will only use the equality (\ref{equality_from_P1}) and Properties (\ref{lemma2holds}-\ref{no-dense-subgraphs}) of the pair $(A,B)$.\\ 

Without loss of generality we may assume that in the second round Spoiler chooses a vertex $x_2 \in V(A)$, since the conditions on $A$ and $B$ in the definition of $\mathcal{P}$ are symmetrical. We examine several cases.\\

\i \label{case1} Let $x_2$ be in a subgraph $U_A \in \mathcal{U}_i(x_1)$ for some $i \geqslant 1$. Due to the equality $\J^i(x_1) = \J^i(y_1)$, the graph $B$ contains an $y_1$-bad subgraph $U_B$ isomorphic to $G_i$ such that
$$T\left(U_A\right) = T\left(U_B\right).$$ 
Let $\varphi: V(U_A) \rightarrow V(U_B)$ be an isomorphism between $U_A$ and $U_B$ that preserves the canonical order. Duplicator answers by choosing $y_2 = \varphi(x_2) \in V(B)$.\\ 

Duplicator's further strategy can be described as follows. If Spoiler chooses a vertex in $U_A$ or $U_B$, Duplicator responds by selecting the corresponding vertex (in respect to $\varphi$) in the other graph. If Spoiler chooses a vertex outside the graphs $U_A$ and $U_B$, Duplicator selects a vertex that not only provides isomorphism between the resulting graphs induced on the sets of selected vertices, but also is in the same adjacency relations with corresponding (in respect to $\varphi$) vertices of the graphs $U_A$ and $U_B$. This property of Duplicator's strategy allows us to assume that all subsequent moves will be made only outside the graphs $U_A$ and $U_B$, since the vertices of these graphs will not affect the Duplicator's victory.

In the third round, Spoiler, without loss of generality, chooses a vertex $x_3 \in V(A) \setminus V(U_A)$. Note that $x_3$ cannot be adjacent to $2$ or more vertices in $V(U_A)$: otherwise, $\left(A|_{V(U_A) \cup \set{x_3}},  U_A\right)$ is $1$-bad \textcolor{black}{and $A|_{V(U_A) \cup \set{x_3}}$ should have a copy in $\mathcal{G}$}, which contradicts the $\G$-maximality of $U_A$.\\

We now examine several subcases of Case \ref{case1}.\\

\pu
Assume that $x_3$ is adjacent to exactly one vertex in $U_A$, say,  $z_{2A}$. Note that a vertex from $A \setminus U_A$ cannot be adjacent to \textcolor{black}{two or more vertices} of $U_A$: otherwise, in $A$ there is a $1$-bad extension of $U_A$ \textcolor{black}{that should therefore have a copy in $\mathcal{G}$}, which contradicts the $\G$-maximality of $U_A$. 

Denote $S = \set{v \in V(U_A) \mid \delta^{U_A}(v, x_3) = 1}$. Note that $|S| \leqslant 2$: otherwise, $x_3$ \textcolor{black}{and the vertices from $\cup_{v\in S}\mathcal{N}^{U_A}(x_3, v)$} form a $1$-bad extension of $U_A$, which contradicts the $\G$-maximality of $U_A$ in $A$.\\

\subpu
Assume that $|S| = 2$. Put $S = \set{z_{1A}, z_{3A}}$. Let also $t_{1A}$, $t_{3A}$ be the vertices adjacent with $z_{1A}$ and $z_{3A}$ respectively, and both adjacent with $x_3$. Note that $t_{1A}\neq t_{3A}$ since otherwise $x_3$, $t_{1A}=t_{3A}$ form a $1$-bad extension of $U_A$. \textcolor{black}{There also cannot be an edge between $t_{1A}$ and $t_{3A}$, for the same reason.} Therefore, $\left(A|_{\set{z_{1A}, z_{2A}, z_{3A}, t_{1A}, x_3, t_{3A}}}, A|_{\set{z_{1A}, z_{2A}, z_{3A}}}\right)$ is generically 
isomorphic to $(K_1, T_1)$. Therefore, 
$\zeta_1^{U_A}({z_{1A}, z_{2A}, z_{3A}}) = 1 = \zeta_1^{U_B}(z_{1B}, z_{2B}, z_{3B})$, 
where $z_{iB} = \varphi(z_{iA})$. Hence, in $B \setminus U_B$ there exist 3 vertices $t_{1B}, y_3, t_{3B}$ forming a strict generalised $(K_1, T_1)$-extension of the vertices $z_{1B}, z_{2B}, z_{3B}$. Thus, Duplicator chooses $y_3 \in V(B)$ in round 3. Note that for each $z_A \in U_A$ and for its image $z_B = \varphi(z_A)$ we have $x_3 \sim z_{A} \Longleftrightarrow y_3 \sim z_{B}$, since $y_3$ has no more than one neighbour in $U_B$ due to $\G$-maximality of $U_B$ in $B$. Moreover, $y_3$ has no more than 2 common neighbours with vertices of $U_B$, also due to $\G$-maximality of $U_B$. 

Consider the fourth round of the Ehrenfeucht game. Without loss of generality, in round 4 Spoiler chooses a vertex $x_4 \in V(A)$. If $x_4$ is adjacent to each of the vertices chosen before then $x_4 \in V(U_A)$ and this case has already been analyzed at the beginning of the reasoning. If $x_4$ is adjacent to exactly 2 of the vertices chosen before and $x_4 \not\in V(U_A)$ then $x_4$ is adjacent to $x_3$ and some $x_i$. Then $x_4$ is either $t_{1A}$ or $t_{3A}$. Duplicator chooses $t_{1B}$ or $t_{3B}$ as $y_4$, and wins. If $x_4$ is adjacent to no more than one vertex chosen before then Duplicator wins due to $3/5$-safeness of the pair $\left(A|_{\set{x_1,x_2,x_3,x_4}}, A|_{\set{x_1,x_2,x_3}}\right)$ and due to Property \ref{end_conds_XY} of $\mathcal{P}$ (\textcolor{black}{in this case, the fact that the extension provided by Property \ref{end_conds_XY} is generic, is not necessary}).\\

\subpu
Assume that $|S| = 1$. Denote $S = \set{z_{1A}}$. Let $t_A \in V(A) \setminus V(U_A)$ be a vertex adjacent to $x_3$ and $z_{1A}$. Obviously, $t_A$ has only one neighbour in $U_A$. 

Consider the pair $\left( A|_{V(U_A) \cup \set{x_3, t_A}}, U_A \right)$. This pair is $3/5$-safe, with $v(A|_{V(U_A) \cup \set{x_3, t_A}},U_A)=2$. By Property \ref{end_conds_XY} of $\mathcal{P}$, in $B$ there is a strict \textcolor{black}{$1$}-generic $\left( A|_{V(U_A) \cup \set{x_3, t_A}}, U_A \right)$-extension of $U_B$. Denote by $y_3$ and $t_B$ the images of $x_3$ and $t_A$ under the isomorphism of the pairs. Duplicator chooses $y_3$. \textcolor{black}{Note that, since $(K^*, T^*)$ is $\alpha$-rigid, we avoid $(K^*, T^*)$-extensions of $U_B$ that have an edge to $x_3$ (due fact that $U_B$ is a \textcolor{black}{$1$}-generic extension). Thus, $\delta^{U_A}(x_3, z_A) = \delta^{U_B}(y_3, \varphi(z_A))$ for each $z_A \in U_A$.}

Consider the last round of the Ehrenfeucht game. Without loss of generality, Spoiler chooses a vertex $x_4 \in V(A)$. If $x_4 \in V(U_A)$ then Duplicator chooses $\varphi(x_4) \in V(U_B)$ and wins. 

If $x_4 \not\in V(U_A)$ then $x_4$ has no more than one neighbour in $U_A$. If, in addition, $x_4$ is not adjacent to $x_3$ or does not have neighbours in $U_A$, Duplicator has the winning move due to $3/5$-safeness of the pair $\left(A|_{\set{x_1,x_2,x_3,x_4}}, A|_{\set{x_1,x_2,x_3}}\right)$. \textcolor{black}{If $x_4$ is adjacent to $x_3$ and has a single neighbour in $U_A$, then this neighbour is $z_{1A}$. Duplicator then chooses $y_4 = t_B$ and wins.} \\

\subpu Assume that $|S| = 0$. Due to Property \ref{end_conds_XY} of $\mathcal{P}$, Duplicator is able to choose a vertex $y_3 \in V(B)$ such that $y_3$ has the only neighbour $z_{2B} = \varphi(z_{2A})$ and $y_3$ does not have common neighbours with vertices from $U_B$. Due to Property \ref{end_conds_XY} of $\mathcal{P}$, the Duplicator's strategy in the last round is trivial. \\

\pu
\label{1b} Assume that $x_3$ is not adjacent to any of vertices of $U_A$. Our goal is to find a vertex $y_3 \in V(B) \setminus V(U_B)$ such that $\delta^{U_A}(x_3, x_i) = \delta^{U_B}(y_3, y_i)$, $i = 1, 2$. 

Define a set of vertices $T$. First of all, $x_3 \in T$. Secondly, for $i = 1, 2$ we do the following: if $\delta^{U_A}(x_3, x_i) = 1$, we add one of the vertices of $V(A) \setminus V(U_A)$ adjacent to $x_3$ and $x_i$ into $T$. Thus, $|T| \leqslant 3$. Note that $T$ does not contain vertices adjacent to at least two vertices of $U_A$, since $T \cap V(U_A) = \varnothing$. 

Consider a graph $\tilde U_A$ such that $V(\tilde U_A) = V(U_A) \cup T$ with the edges from $A$ except for the possible edge between the vertices from $T \setminus \set{x_3}$. The pair $(\tilde U_A, U_A)$ is $3/5$-safe, since $x_3$ does not have neighbours in $U_A$. Therefore, due to Property \ref{end_conds_XY} of $\mathcal{P}$, in $B$ there is a \textcolor{black}{strict $1$-generic} $(\tilde U_A, U_A)$-extension of $U_B$. Duplicator chooses a vertex $y_3 \in V(B)$ that is the image of $x_3$ under the isomorphism of the pairs. For $y_3$, obviously, the equality $\delta^{U_A}(x_3, x_i) = \delta^{U_B}(y_3, y_i)$ holds for $i = 1, 2$, \textcolor{black}{due to the fact that there are no $(K^*,T^*)$-extensions of $U_B$ that have an edge to $x_3$.}

Consider the last round of the Ehrenfeucht game. Without loss of generality, Spoiler chooses a vertex $x_4 \in V(A)$. If $x_4 \in V(U_A)$ then Duplicator chooses $y_4 = \varphi(x_4) \in V(U_B)$ and wins. Otherwise, $x_4$ has no more than one neighbour among $x_1$, $x_2$. If, in addition, $x_4$ is not adjacent to $x_3$ or to both $x_1,x_2$, Duplicator has a winning move due to $3/5$-safeness of the corresponding pair (\textcolor{black}{which can be easily verified}). 
Finally, if $x_4$ is adjacent to $x_3$ and has a single neighbour among $x_1$ and $x_2$, Duplicator wins due to the equalities $\delta^{U_A}(x_3, x_i) = \delta^{U_B}(y_3, y_i)$ with $i = 1, 2$.

Case \ref{case1} is finished.\\

\i \label{case2}
Let $x_2$ be a vertex in the $0$-neighbourhood of $x_1$ such that there is an $x_1$-bad subgraph $U_A$ with the following condition: for some $j \in \set{1, 2}$ and some tuple $\tilde T$ of $\nu_j$ vertices from $U_A$ there is an extension $Q_A \supset U_A$ such that $x_2 \in V(Q_A, U_A)$ and $Q_A|_{[V(Q_A)\setminus V(U_A)]\cup\tilde T}$ is a generalised $(K_j,T_j)$-extension of $\tilde T$. There can only be one such subgraph by Property \textcolor{black}{\ref{two-bad-subgraphs-cant-intersect-kjtj} of $\mathcal{G}$}.

Due to the choice of $y_1$, in $B$ there are induced subgraphs $U_B \subset Q_B$ such that $U_B$ is $y_1$-bad, $T(U_A) = T(U_B)$, $(Q_A, U_A)$ is isomorphic to $(Q_B, U_B)$ (under some isomorphism $\varphi$ that maps $x_1$ to $y_1$ and preserves the canonical order of $U_A$ and $U_B$), and $Q_A$ and $Q_B$ have identical generalised $(K_1,T_1)$-extensions in $A$ and $B$, respectively \textcolor{black}{(i.e. $\theta^{U_A}(Q_A) = \theta^{U_B}(Q_B)$, recall that the definition of $\theta^U(K)$ is given before the statement of Lemma \ref{lemma1})}. Duplicator chooses $y_2 = \varphi(x_2)$.

Note that there is no vertex in $A \setminus Q_A$ adjacent to at least two vertices from $Q_A$: \textcolor{black}{the existence of such a vertex contradicts $\mathcal{G}$-maximality of $U_A$.} The same holds for $B$. Without loss of generality, Spoiler chooses $x_3 \in V(A)$ in the third round.  

\pu\label{subc:3a_i} Assume that $x_3\notin V(Q_A)$. 

\subpu Assume first that $x_3$ is adjacent to a vertex $u$ from $Q_A$. If $\delta^{Q_A}(x_1, x_3) = \delta^{Q_A}(x_2,x_3) = 1$, let $t_{iA}$ be a vertex from $\mathcal{N}^{Q_A}(x_i, x_3)$ for $i = 1,2$. The pair $(A|_{V(Q_A) \cup \set{x_3, t_{1A}, t_{2A}}}, Q_A)$ is generically isomorphic to $(K_1, T_1)$. Therefore, $Q_B$ has an identical generalised  $(K_1,T_1)$-extension, \textcolor{black}{since $\theta^{U_A}(Q_A) = \theta^{U_B}(Q_B)$}. In other words, there exists $y_3 \in V(B) \setminus V(Q_B)$ adjacent to $\varphi(u)$ such that $\delta^{Q_B}(y_1, y_3) = \delta^{Q_A}(y_2,y_3) = 1$. Duplicator chooses $y_3$ and wins. 

\textcolor{black}{Consider the case $\delta^{Q_A}(x_1, x_3) = 0$ or $\delta^{Q_A}(x_2,x_3) = 0$. Define a set of vertices $T$ in a \textcolor{black}{similar} way to that in Subcase \ref{1b}: $T$ contains $x_3$, one vertex from $\mathcal{N}^{Q_A}(x_1, x_3)$ if it is non-empty, and one vertex from $\mathcal{N}^{Q_A}(x_2, x_3)$ if it is non-empty. Since both $\mathcal{N}^{Q_A}(x_1, x_3)$ and $\mathcal{N}^{Q_A}(x_2, x_3)$ cannot be non-empty, $|T| \leqslant 2$. Note that a vertex from $T$ cannot be adjacent to more than one vertex of $Q_A$ since $T \cap V(Q_A) = \varnothing$. Consider a graph $\tilde Q_A$ such that $V(\tilde Q_A) = V(Q_A) \cup T$ with the edges induced by $A$. The pair $(\tilde Q_A, Q_A)$ is $3/5$-safe. Therefore, due to Property \ref{end_conds_XY} of $\mathcal{P}$, in $B$ there is a strict \textcolor{black}{$1$-generic} $(\tilde Q_A, Q_A)$-extension of $Q_B$. Duplicator chooses $y_3 \in V(B)$ as the image of $x_3$ under the isomorphism of the pairs. For $y_3$, the equality $\delta^{Q_A}(x_3, x_i) = \delta^{Q_B}(y_3, y_i)$ holds for $i = 1, 2$, since there are no $(K^*, T^*)$-extensions of $Q_B$ that have an edge to $x_3$.}

\textcolor{black}{Consider the last round of the Ehrenfeucht game. The existence of the last move by Duplicator in this case can be verified similarly to Subcase \ref{1b}. }

\subpu If $x_3$ is not adjacent to any vertex from $Q_A$, then Duplicator chooses a vertex $y_3\notin V(Q_B)$ that is not adjacent to any vertex of $Q_B$ and ensures that $x_3$ has a common neighbour with $x_j$, $j\in\{1,2\}$, if and only if $y_3$ has a common neighbour with $y_j$, $j\in\{1,2\}$ (\textcolor{black}{It can be achieved similarly to Subcase \ref{1b}} since $x_3,x_1,x_2$ do not have a common neighbour and due to Property~\ref{end_conds_XY} of $\mathcal{P}$). If Spoiler in the fourth round chooses a vertex in $V_A$ or $V_B$, then Duplicator chooses the respective vertex in the other graph. If he chooses a vertex outside $V_A$ (say, he plays in $A$), then it is either adjacent to at most one vertex of $x_1,x_2,x_3$ (and then Duplicator wins due to the Property 4 of $\mathcal{P}$), or it is adjacent to $x_3$ and one vertex of $x_1,x_2$, which is also foreseen by Duplicator.\\


\pu\label{subc:3a_ii} Consider the case $x_3\in V(Q_A)$. Duplicator chooses $y_3=\varphi(x_3)$. Since any vertex outside $Q_A$ ($Q_B$) is adjacent to at most one vertex of $Q_A$ ($Q_B$), Duplicator wins in the fourth round by Property~\ref{end_conds_XY} of $\mathcal{P}$.  \\

\i \label{case3} Let $x_2$ be a vertex in the $0$-neighbourhood of $x_1$ \textcolor{black}{such that there is no $x_1$-bad subgraph $U_A$ and extension $Q_A$ described in Case \ref{case2} that contains $x_2$}, 
and let $x_2$ do not have neighbours in $x_1$-bad subgraphs (except for, probably, \textcolor{black}{$x_1$}). 

By Property \ref{lemma2holds} of $\mathcal{P}$, there is $y_2 \in V(B)$ such that the conditions of Lemma \ref{lemma2} hold. Duplicator chooses such a vertex $y_2$. Let $W_A$ and $W_B$ be the $3$-order $(K^*, T^*)$-neighbourhoods of $A|_{\{x_1,x_2\}}$ and $B|_{\{y_1,y_2\}}$ respectively. Note that, \textcolor{black}{by Item \ref{lemma2-end-conclusion} of the conclusion of Lemma \ref{lemma2}}, there exists an isomoprphism $\varphi$ between $W_A$ and $W_B$ such that $\varphi(x_1)=y_1$ and $\varphi(x_2)=y_2$.

Without loss of generality we may assume that in the third round Spoiler chooses a vertex $x_3 \in V(A)$. We examine several subcases of Case \ref{case3}.\\

\pu Assume that $x_3$ is in the $2$-order $(K^*, T^*)$-neighbourhood of $A|_{\{x_1,x_2\}}$. Duplicator chooses $y_3=\varphi(x_3)$. 
If, in the fourth round, Spoiler chooses a vertex \textcolor{black}{say, $x_4 \in V(A)$} that is adjacent to at least two of the previous vertices, \textcolor{black}{then $x_4 \in W_A$. Since $W_A$ and $W_B$ are isomorphic, Duplicator chooses $y_4 = \varphi(x_4)$ and wins}. If Spoiler chooses a vertex adjacent to at most one of the previous vertices, the last step is obvious due to the fact \textcolor{black}{that an extension with one edge is always $3/5$-safe.} \\

\pu Assume that $x_3$ is not in the $2$-order $(K^*, T^*)$-neighbourhood of $A|_{\set{x_1, x_2}}$ and $x_3$ is not adjacent both to $x_1$ and $x_2$. \\

\subpu Assume that $\mathcal{N}(x_1,x_2,x_3)$ is not empty. If $|\mathcal{N}(x_1,x_2,x_3)| \geqslant 2$ then $x_3$ is in the $2$-order $(K^*, T^*)$-neighbourhood of $A|_{\set{x_1, x_2}}$, which contradicts the assumption. Thus, we denote $t_A$ as the only vertex in $\mathcal{N}(x_1,x_2,x_3)$. Due to the isomorphism of $W_A$ and $W_B$, there is $t_B \in V(B)$ adjacent to $y_1$, $y_2$.

We now need to select a vertex $y_3 \in V(B)$ such that $y_3 \sim t_B$, $\delta(x_1,\lnot x_2,x_3) = \delta(y_1,\lnot y_2,y_3)$ and $\delta(\lnot x_1,x_2,x_3) = \delta(\lnot y_1,y_2,y_3)$. \textcolor{black}{If $\delta(x_1,\lnot x_2,x_3) = \delta(\lnot x_1,x_2,x_3) = 1$ then $\zeta_1^{W_A}(x_1,x_2,t_A)=1= \zeta_1^{W_B}(y_1,y_2, t_B)$. We therefore can select $y_3$ from the corresponding $(K_1,T_1)$-extension of $y_1$, $y_2$, $t_B$. 
\textcolor{black}{If $\delta(x_1,\lnot x_2,x_3) \neq 1$ or $\delta(\lnot x_1,x_2,x_3) \neq 1$, let $K$ be the subgraph of $X$ induced on $x_1$, $x_2$, $x_3$, $t_A$ and a vertex from $\mathcal{N}(x_1,\lnot x_2,x_3)$ or $\mathcal{N}(\lnot x_1,x_2,x_3)$ if one of the latter is non-empty. The pair $(K, K|_{\set{x_1,x_2,t_A}})$ is clearly $3/5$-safe. Therefore,} by Property \ref{end_conds_XY} of $\mathcal{P}$, there is a strict \textcolor{black}{$1$}-generic \textcolor{black}{$(K, K|_{\set{x_1,x_2,t_A}})$-extension of $Y|_{\set{y_1,y_2,t_B}}$} that guarantees the existence of $y_3$ such that $\delta(x_1,\lnot x_2,x_3) = \delta(y_1,\lnot y_2,y_3)$ and $\delta(\lnot x_1,x_2,x_3) = \delta(\lnot y_1,y_2,y_3)$. Duplicator chooses such a vertex $y_3$ and wins in the fourth round: the only non-trivial case is when Spoiler chooses a vertex from $\mathcal{N}(x_1,x_2,\lnot x_3)$, but then $\mathcal{N}(y_1,y_2,\lnot y_3)$ is also non-empty, due to the isomorphism of the $2$-order $(K^*, T^*)$-neighbourhoods of $A|_{\{x_1, x_2\}}$, $B|_{\{y_1, y_2\}}$  and the equality $|\mathcal{N}(x_1,x_2,x_3)| = |\mathcal{N}(y_1,y_2,y_3)|=1$.\\}

\subpu Assume that $\mathcal{N}(x_1,x_2,x_3)$ is empty. Then Duplicator chooses a vertex $y_3 \in V(B)$ such that $\mathcal{N}(y_1,y_2,y_3)$ is empty, $\delta(x_1,\lnot x_2,x_3) = \delta(y_1,\lnot y_2,y_3)$ and $\delta(\lnot x_1,x_2,x_3) = \delta(\lnot y_1,y_2,y_3)$ (such $y_3$ exists due to Property \ref{end_conds_XY} of $\mathcal{P}$). \textcolor{black}{Clearly, $\delta(x_1,x_2,\lnot x_3) = \delta(y_1,y_2,\lnot y_3)$, due to the isomorphism of the $2$-order $(K^*, T^*)$-neighbourhoods of $A|_{\{x_1, x_2\}}$, $B|_{\{y_1, y_2\}}$. The further winning strategy is obvious.}\\

\pu Assume that $x_3$ is not in the $2$-order $(K^*, T^*)$-neighbourhood of $A|_{\set{x_1, x_2}}$ and $x_3$ is adjacent to exactly one of $x_1$ and $x_2$. Note that $\mathcal{N}(x_1,x_2,x_3)$ is empty (otherwise, $x_3$ belongs to the $2$-order $(K^*, T^*)$-neighbourhood of $A|_{\set{x_1, x_2}}$). 

We now need to select a vertex $y_3 \in V(B)$ such that $y_3$ is not in the $2$-order $(K^*, T^*)$-neighbourhood of $B|_{\set{y_1,y_2}}$, and also $\delta(x_1,x_3) = \delta(y_1,y_3)$ and $\delta(x_2,x_3) = \delta(y_2,y_3)$. \textcolor{black}{Let us prove the existense of such $y_3$. If both $\delta(x_1,x_3)$ and $\delta(x_2,x_3)$ are equal to $1$, the existence follows from the fact that $\zeta_1^{W_A}(x_1,x_1,x_2) = \zeta_1^{W_B}(y_1,y_1,y_2)$ and $\zeta_1^{W_A}(x_1,x_2,x_2) = \zeta_1^{W_B}(y_1,y_2,y_2)$.} 

\textcolor{black}{If $\delta(x_1,\lnot x_2,x_3) \neq 1$ or $\delta(\lnot x_1,x_2,x_3) \neq 1$, let $K$ be the subgraph of $X$ induced on $x_1$, $x_2$, $x_3$ and a vertex from $\mathcal{N}(x_1,\lnot x_2,x_3)$ or $\mathcal{N}(\lnot x_1,x_2,x_3)$ if one of the latter is non-empty. The pair $(K, K|_{\set{x_1,x_2}})$ is clearly $3/5$-safe. Therefore, by Property \ref{end_conds_XY} of $\mathcal{P}$, there is a strict $1$-generic $(K, K|_{\set{x_1,x_2}})$-extension $\hat K$ of $Y|_{\set{y_1,y_2}}$. Duplicator selects $y_3$ to be the image of $x_3$ under the isomorphism of $K$ and $\hat K$. Since $\hat K$ is a $1$-generic extension, we have $\delta(x_1,x_3) = \delta(y_1,y_3)$ and $\delta(x_2,x_3) = \delta(y_2,y_3)$. Lastly, verify that $y_3$ is not in the $2$-order $(K^*, T^*)$-neighbourhood of $B|_{\set{y_1,y_2}}$. Otherwise, it is adjacent to at least one vertex from $\mathcal{N}(y_1,y_2)$, which contradicts the fact that $\hat K$ is a $1$-generic extension of $Y|_{\set{y_1,y_2}}$.}

Thus, Duplicator selects such a vertex $y_3$. Clearly, Duplicator wins in the fourth round. Indeed, the only non-trivial case is if Spoiler selects a common neighbour of $x_1,x_2$ (without loss of generality, he plays in $A$). Since $W_A$ and $W_B$ are isomorphic, there exists a common neighbour of $y_1,y_2$ in $B$. It is not adjacent to $y_3$ since otherwise $y_3$ is in the $2$-order $(K^*, T^*)$-neighbourhood of $B|_{\set{y_1, y_2}}$.\\


\i \label{case4} Let $x_2$ be a vertex in the $0$-neighbourhood of $x_1$ \textcolor{black}{such that there is no $x_1$-bad subgraph $U_A$ and extension $Q_A$ described in Case \ref{case2} that contains $x_2$}, and $x_2$ has a neighbour in at least one of $x_1$-bad subgraphs (excluding $x_1$). 

Let us first prove that such a neighbour is unique. Indeed, if there is an $x_1$-bad subgraph $U$ with two neighbours of $x_2$ then $U$ is not $\G$-maximal. If there are two $x_1$-bad subgraphs $U_1$ and $U_2$ with vertices $u_1 \in V(U_1)$ and $u_2 \in V(U_2)$ such that $u_1$ and $u_2$ are both neighbours of $x_2$ then, \textcolor{black}{by Property \ref{begin_conds_XY} of $\mathcal{P}$} and Property \ref{two-bad-subgraphs-cant-intersect-kjtj} \textcolor{black}{(Corollary, Item 1)} of $\G$, 
\textcolor{black}{either the graph induced on $V = V(U_1) \cup V(U_2) \cup \set{x_2}$ has a copy in $\G$, which contradicts the fact that $U_1$ and $U_2$ are $\G$-maximal}, or $\rho^{\max}(A|_{V}) \geqslant 5/3$ with $|V| \leqslant 352$, which contradicts Property \ref{no-dense-subgraphs} of $\mathcal{P}$. Thus, there is only one neighbour of $x_2$ in a single $x_1$-bad subgraph $U_A$. 

\textcolor{black}{Moreover, note that $x_1 \not\sim x_2$: otherwise $x_2$ has at least two neighbours in $U_A$, which contradicts the fact that $U_A$ is $\mathcal{G}$-maximal.
}

Note that the if-conditions 1--5 of Lemma~\ref{lemma3} hold. 
By Property~\ref{lemma3holds} of $\mathcal{P}$ and due to the choice of $y_1$, in $B$ there is a vertex $y_2$ such that the conditions of Lemma~\ref{lemma3} hold. Duplicator chooses $y_2$. Let $U_B$ be the $x_1$-bad subgraph in $B$ that corresponds to $U_A$. Let $W_A$ and $W_B$ be the 1-order $(K^*, T^*)$-neighbourhoods of $A|_{V(U_A) \cup \set{x_2}}$ and $B|_{V(U_B) \cup \set{y_2}}$ in $A$ and $B$, and let $\varphi$ be an isomorphism between $W_A$ and $W_B$ that maps $x_1$ to $y_1$, $x_2$ to $y_2$ and preserves the canonical order of $U_A$ and $U_B$. Note that $|V(W_A) \setminus (V(U_A) \cup\{x_2\})| \leqslant 1$ (see Remark after Lemma~\ref{lemma3}).

Without loss of generality, in the third round Spoiler chooses a vertex $x_3 \in V(A)$. The case when $x_3$ belongs to $W_A$ is the same as Subcase~2b since, if $x_4$ is outside $W_A$, then it may have at most one neighbour in $W_A$ as well. 

Consider the case when $x_3$ is \textcolor{black}{not} in $W_A$. \textcolor{black}{Let us prove that $x_3$ cannot be adjacent to $2$ or more vertices from $W_A$. Assume the contrary. Note that $x_3$ cannot have two neighbours in $V(U_A)\cup\{x_2\}$ since $x_3 \not \in W_A$. Then $x_3$ is adjacent to the only vertex from $V(W_A) \setminus (V(U_A)\cup\{x_2\})$ (denoted by $t_A$) which is, in turn, adjacent to $x_2$ and a single vertex from $U_A$.} 
If $x_3$ is adjacent to $x_2$ then $(A|_{V(U_A) \cup \set{x_2, x_3, t_A}} , U_A)$ is generically isomorphic to \textcolor{black}{$(K_2,T_2)$, which contradicts the if-condition \ref{l3_end_cond} of Lemma \ref{lemma3}. Similarly, if $x_3$ is adjacent to a vertex from $U_A$ then $(A|_{V(U_A) \cup \set{x_2, x_3, t_A}}, U_A)$ is generically isomorphic to $(K_1,T_1)$, which again contradicts the if-condition \ref{l3_end_cond} of Lemma \ref{lemma3}. Thus, $x_3$ can be adjacent to at most one vertex from $W_A$. Further reasoning is the same as in Subcase 2a.}

Thus, we have considered all possible cases and established that Duplicator has a winning strategy in all of them. The proof of Theorem \ref{main_theorem} is complete.

%
%
%


\section{Conclusion}
{\color{black}We have established that $3/5$ is not a limit point of the $4$-spectrum. The $4$-spectrum is thus a finite set of rational numbers in $[\frac 12, 2]$ (due to Theorems \ref{alpha irrational} and \ref{zhukovskiiless}), and $k=5$ is the minimum quantifier depth such that the $k$-spectrum is infinite. Moreover, $\frac 12$ is in the $4$-spectrum by Theorem \ref{zhukovskiiless}. It would be of interest to describe the 4-spectrum completely. Note that some points of the 4-spectrum are known.
For instance, from the proof of Theorem~\ref{alpha irrational} for $G(n,n^{-1-1/\ell})$ it follows that $2, 3/2, 4/3, \ldots, 1+1/m$ are in the $4$-spectrum and $1+1/(m+1)$ is not in the $4$-spectrum for some $m \in \mathbb{N}$ (it is not hard to see that $m\geq 20$). However, the exact value of such $m$ is unknown. Moreover, the minimal element of the $4$-spectrum greater than $1$ is also unknown (it is also of the form $1+1/\hat m$ by Theorem~\ref{alpha irrational}, but $\hat m$ could be much bigger than $m$).

For $\alpha < 1$, some results were obtained in \cite{universal-law,smallk,4-law,zh_sbornik}. In particular, the largest $\alpha < 1$ in the $4$-spectrum is $13/14$~\cite{zh_sbornik}. Moreover, $3/5$, $1-1/m$ for $m \in \{3,8,9,10,11,12,13\}$ are in the $4$-spectrum~\cite{4-law}. Meanwhile, some intervals, in particular, $(53/80, 2/3)$, $(7/8,43/49)$ (and many others) do not intersect with the $4$-spectrum~\cite{universal-law,smallk}. 

A systematic approach to fully explain the structure of the $4$-spectrum (or a $k$-spectrum for $k \geqslant 4$) has not yet been proposed, and it appears to be a very hard problem.}

\section{Acknowledgements}

This work was supported by a grant for research centers in the field of artificial intelligence, provided by the Analytical Center for the Government of the Russian Federation in accordance with the subsidy agreement (agreement identifier 000000D730321P5Q0002) and the agreement with the Moscow Institute of Physics and Technology dated November 1, 2021 No. 70-2021-00138.

\bibliographystyle{plain}
\bibliography{bibliography}

\end{document}